\documentclass[a4paper,12pt]{article}
\PassOptionsToPackage{dvipsnames}{xcolor}

\usepackage[utf8]{inputenc}
\usepackage[english]{babel}

\usepackage{subcaption}
\usepackage{datetime}
\usepackage{enumerate}

\usepackage{amsfonts,amsmath,amssymb,amsbsy,amsthm}
\usepackage{mathtools}
\usepackage[T1]{fontenc}
\usepackage{lmodern}

\usepackage{floatrow}

\usepackage{verbatim}

\usepackage[backref=page]{hyperref}
\hypersetup{%
	colorlinks,
	linkcolor={red!60!black},
	citecolor={green!60!black},
	urlcolor={blue!60!black}
}

\usepackage{tikz}
\usetikzlibrary{calc,positioning,decorations.pathmorphing,decorations.pathreplacing}
\tikzset{emp/.style={double distance = 0.3ex}}
\tikzset{oriented/.style={->,shorten >= 1.5pt}}
\usepackage{./Figures/figstyle}
\usepackage{pgfplots}

\usepackage{lineno}
\newcommand*\patchAmsMathEnvironmentForLineno[1]{%
	\expandafter\let\csname old#1\expandafter\endcsname\csname #1\endcsname
	\expandafter\let\csname oldend#1\expandafter\endcsname\csname end#1\endcsname
	\renewenvironment{#1}%
	{\linenomath\csname old#1\endcsname}%
	{\csname oldend#1\endcsname\endlinenomath}}%
\newcommand*\patchBothAmsMathEnvironmentsForLineno[1]{%
	\patchAmsMathEnvironmentForLineno{#1}%
	\patchAmsMathEnvironmentForLineno{#1*}}%
\AtBeginDocument{%
	\patchBothAmsMathEnvironmentsForLineno{equation}%
	\patchBothAmsMathEnvironmentsForLineno{align}%
	\patchBothAmsMathEnvironmentsForLineno{flalign}%
	\patchBothAmsMathEnvironmentsForLineno{alignat}%
	\patchBothAmsMathEnvironmentsForLineno{gather}%
	\patchBothAmsMathEnvironmentsForLineno{multline}%
}

\usepackage[abbrev,msc-links,backrefs]{amsrefs}
\usepackage{doi}

\renewcommand{\PrintDOI}[1]{\doi{#1}}

\newtheorem{theorem}{Theorem}
\newtheorem{conjecture}[theorem]{Conjecture}
\newtheorem{lemma}[theorem]{Lemma}

\newtheorem{claim}{Claim}
\newtheorem{subclaim}{Subclaim}[claim]
\newtheorem{definition}{Definition}

\newcommand{\pn}{{\rm pn}}

\newcommand{\D}{\mathcal{D}}

\newcommand{\Hcal}{\mathcal{H}}
\newcommand{\Lcal}{\mathcal{L}}
\newcommand{\Acal}{\mathcal{A}}

\newcommand{\floor}[1]{\lfloor #1 \rfloor}

\newcommand{\mydiff}{%
  \mathrel{\vbox{\offinterlineskip\ialign{%
    \hfil##\hfil\cr
    $\scriptscriptstyle\ast$\cr
    $-$\cr
}}}}

\title{Gallai's path decomposition conjecture \\for triangle-free planar graphs}

\author{F. Botler\textsuperscript{1} \hspace{.5cm} A. Jim\'enez\textsuperscript{1} \hspace{.5cm} M. Sambinelli\textsuperscript{2}\\
	{\footnotesize \textsuperscript{1}CIMFAV, Facultad de Ingeniería}\vspace{-.2cm}\\
	{\footnotesize Universidad de Valpara\'iso}\\
	{\footnotesize \textsuperscript{2}Institute of Computing}\vspace{-.2cm}\\
	{\footnotesize University of Campinas}
	\footnote{
This research was partially supported by 
Millenium Nucleus Information and Coordination in Networks (ICM/FIC RC 130003).
F. Botler is partially supported by CONICYT/FONDECYT/POSTDOCTORADO 3170878.
A. Jim\'enez is partially supported by  CONICYT/FONDECYT/INICIACION 11170931.
  E-mails:
  fbotler@dii.uchile.cl (F. Botler),
  andrea.jimenez@uv.cl (A. Jim\'enez)
  msambinelli@ic.unicamp.br (M. Sambinelli).
}}

\shortdate
\yyyymmdddate
\settimeformat{ampmtime}
\date{\today, \currenttime}
\date{}
\begin{document}

\sloppy


\maketitle

\begin{abstract}
	A path decomposition of a graph \(G\) is a collection of edge-disjoint paths of \(G\) that covers the edge set of \(G\).
	Gallai (1968) conjectured that every connected graph on~\(n\) vertices admits a path decomposition of cardinality at most \(\lfloor (n+1)/2\rfloor\).
	Gallai's Conjecture has been verified for many classes of graphs.
	In particular, Lovász (1968) verified this conjecture for graphs with at most one vertex with even degree, and Pyber (1996) verified it for graphs in which every cycle contains a vertex with odd degree.
	Recently, Bonamy and Perrett (2016) verified Gallai's Conjecture for graphs with maximum degree at most \(5\), 
	and Botler~et al.~(2017) verified it for graphs with treewidth at most \(3\).
	In this paper, we verify Gallai's Conjecture for triangle-free planar graphs.

\end{abstract}

\section{Introduction}\label{sec:introduction}

All graphs considered here are finite and simple, i.e., contain a finite number of vertices and edges and has  neither loops nor multiple edges.
A \emph{decomposition} \(\D\) of a graph \(G\) is a collection of edge-disjoint subgraphs of \(G\) that covers all the edges of \(G\).
A decomposition \(\D\) is a \emph{path decomposition} if every element in \(\D\) is a path.
A path decomposition~\(\D\) of a graph~\(G\) is \emph{minimum} if for every path decomposition~\(\D'\) of~\(G\) we have \(|\D| \leq |\D'|\).
The cardinality of such a minimum path decomposition is called the \emph{path number} of \(G\) and it is denoted by \(\pn(G)\).
In 1968, Gallai proposed the following conjecture~(see~\cite{Lovasz68, Bondy14}).

\begin{conjecture}[Gallai, 1968]\label{conj:gallai}
The path number of a connected graph on \(n\) vertices is at most $\big\lfloor \frac{n+1}{2} \big\rfloor$.
\end{conjecture}

Lov\'asz~\cite{Lovasz68} verified Conjecture~\ref{conj:gallai} for graphs that have at most one vertex with even degree.
Pyber~\cite{Pyber96} extended Lov\'asz's result by proving that Conjecture~\ref{conj:gallai} holds for graphs whose cycles have vertices  
with odd degree, and, in 2005, Fan~\cite{Fan05} extended it even further.
Following another direction, 
Favaron and Kouider~\cite{FavaronKouider88} verified
Conjecture~\ref{conj:gallai} for Eulerian graphs with maximum degree at most \(4\),
Jim\'enez and Wakabayashi~\cite{JimenezWakabayashi2017} verified it for a family of triangle-free graphs,
and Botler and Jim\'enez~\cite{BotlerJimenez2015+} verified it for a family of even regular graphs with a high girth condition.
Also, Geng, Fang, and Li~\cite{GengFangLi15} verified Conjecture~\ref{conj:gallai} for
maximal outerplanar graphs and \(2\)-connected outerplanar graphs,
and Bonamy and Perrett~\cite{BonamyPerrett16+} verified it for 
graphs with maximum degree at most \(5\).

More recently, Botler~et~al.~\cite{BoSaCoLe} verified Conjecture~\ref{conj:gallai} for graphs with treewidth at most~\(3\),
by proving that a partial \(3\)-tree with \(n\) vertices either has path number at most~\(\lfloor n/2\rfloor\), which are called \emph{Gallai graphs},
or is one of the two exceptions (\(K_3\) and \(K_5-e\)).
They also proved that graphs with maximum degree at most \(4\) are either Gallai graphs,
or one of the three exceptions (\(K_3\), \(K_5-e\), and \(K_5\)).
In this paper, we prove that every triangle-free planar graph is a Gallai graph, with no exceptions (see Theorem~\ref{theo:main}).
This verifies Gallai's Conjecture for these graphs.

Botler et al. introduced the concept of \emph{reducing subgraph},
which plays an important role in their proof.
A reducing subgraph is a forbidden structure in a minimal counterexample for the statement of their main theorem.
Here we introduce the concept of \emph{feasible reducing scheme}, a generalization of reducing subgraph.
We follow the same strategy as in~\cite{BoSaCoLe}: 
we choose a minimal counterexample
and explore the structure around some special vertices in it, which are called \emph{terminals}.
Our proof consists of eleven claims that provide a strict structure for our counterexample.
In particular, we prove that a minimal counterexample for Theorem~\ref{theo:main} must be \(2\)-connected (see Claims~\ref{claim:no-useful-cut-vertex} and~\ref{claim:no-terminal-degree-1}), and every pair of its terminals are separated by a separator with two vertices (see Claim~\ref{claim:terminal-separator}).

This work is organized as follows.
In Section~\ref{sec:notation}, we establish some basic notations, and
present two auxiliary results.
In Section~\ref{sec:reducing}, we define feasible reducing schemes,
and present some tools that allow us to deal with them.
In Section~\ref{sec:main-result}, we settle Conjecture~\ref{conj:gallai} for triangle-free planar graphs.
In Section~\ref{sec:concluding}, we present some concluding remarks.

\section{Notation and auxiliary results}\label{sec:notation}

The basic terminology and notation used in this paper are standard (see, e.g.~\cite{Di10}).
Given a graph \(G\), we denote its vertex set by \(V(G)\) and its edge set by \(E(G)\).
The set of neighbors of a vertex \(v\) in a graph \(G\) is denoted by \(N_G(v)\) and its degree by \(d_G(v)\).
When \(G\) is clear from the context, we simply write \(N(v)\) and  \(d(v)\).
Since \(G\) is simple, we always have \(d_G(v) = |N_G(v)|\).

A graph \(H\) is a \emph{subgraph} of a graph \(G\), denoted by \(H \subseteq G\), if \(V(H) \subseteq V(G)\) and \(E(H) \subseteq E(G)\).
Given a set of vertices \(X \subseteq V(G)\), we say that \(H\) is the subgraph of \(G\) \emph{induced by \(X\)}, 
denoted by \(G[X]\), if  \(V(H) = X\) and \(E(H) = \{xy \in E(G) \colon x,y \in X\}\).
Given a set of edges \(Y \subseteq E(G)\), we say that \(H\) is the subgraph  of \(G\) \emph{induced by \(Y\)}, 
if \(E(H) = Y\) and \(V(H) = \{u \in V(G) \colon uv \in Y\}\).
Given two (not necessarily disjoint) graphs \(G\) and \(H\), we define the graph \(G + H\) 
by \(V(G + H) = V(G) \cup V(H)\) and \(E(G + H) = E(G) \cup E(H)\).
When \(H\) is an edge \(e\), we simply write \(G + e\) to denote \(G+H\).
Given \(X \subseteq V(G)\), we define \(G - X=G[V(G) \setminus X]\).
In the case that \(X = \{v\}\), we simply write \(G - v\).
Given a set~\(Y\) of edges, we define the graphs \(G - Y=(V(G),E(G)\setminus Y)\),  \(G \mydiff Y=(V(G)\setminus I,E(G)\setminus Y)\), where $I$ is the set of isolated vertices of $G-Y$, and \(G + Y = G + \sum_{e \in Y} e\).
As before, in the case that \(Y = \{e\}\), we simply write \(G - e\) or \(G \mydiff e\).

A \emph{path} \(P\) in a graph \(G\) is a sequence \(v_0v_1\cdots v_\ell\) of distinct vertices in \(V(G)\) 
such that~\(v_iv_{i+1}\in E(G)\), for~\(i=0,\ldots,\ell-1\).
We say that \(v_0\) and \(v_\ell\) are the \emph{end vertices} of \(P\),
and that \(P\) \emph{joins} \(v_0\) and \(v_\ell\).
When it is convenient, we consider a path 
as the subgraph of~\(G\) induced by the set of edges \(\{v_iv_{i+1} \colon i = 0, \ldots, \ell - 1\}\).
The length of a path is defined as its number of edges.
Given two vertices $u,v \in V(G)$, we denote by \(dist_{G}(u,v)\) the minimum length of a path that joins
$u$ and $v$.

Given a path decomposition \(\D\) of a graph \(G\)
and a vertex \(v\in V(G)\), we denote by \(\D(v)\) the number of paths in \(\D\) that have \(v\) as an end vertex.
It is not hard to check that \(\D(v) \equiv d(v) \pmod{2}\).
In particular, if \(d_G(v)\) is odd, we have \(\D(v) \geq 1\),
for any path decomposition \(\D\) of \(G\).

Figures in this paper are depicted as follows (see Figure~\ref{fig:example}).
We denote vertices by circles or squares.
A circle illustrates a general vertex while a square illustrates a vertex where all edges incidents to it are present in the figure.
Straight and curved lines are used to illustrate simple edges, while snaked lines are used to illustrate paths (possibly with internal vertices).

\begin{figure}[h]
	\floatbox[{\capbeside\thisfloatsetup{capbesideposition={right,center},capbesidewidth=10cm}}]{figure}[\FBwidth]
	{\caption{The circles $a$ and $d$ illustrate general vertices, while the squares \(b\) and \(c\) illustrate vertices for which all edges
	 incident to them are present in the figure (this is useful for pictures in which just a fraction of the graph appears). 
	 The lines between the pairs of vertices \((a, c), (c, b), (d, b), (a, d)\) illustrate simple edges between these pair of vertices, 
	 while the snaked lines between the pairs \((a, b), (c, d)\) illustrate paths joining these pair of vertices.
		}\label{fig:example}}
	{\scalebox{.8}{\begin{tikzpicture}[scale = 0.7]

	\tikzstyle{every circle node} = [draw];

	\node (x1) [black vertex] at (90:3) {};
	\node (x2) [squared black vertex] at (210:3) {};
	\node (x3) [squared black vertex] at (0,0) {};
	\node (x4) [black vertex] at (-30:3) {};
	\node () [] at (0:4) {};
	
	\node (label_a) [] at (90:3.5) {$a$};
	\node (label_b) [] at (210:3.6) {$b$};
	\node (label_c) [] at ($(x3)+(45:.6)$) {$c$};
	\node (label_d) [] at (-30:3.6) {$d$};

        \draw[edge] (x3) -- (x1);
        \draw[edge, short snake] (x1) to [bend right] (x2);
        \draw[edge] (x2) -- (x3);

	\draw[edge] (x2) to [bend right] (x4);
	\draw[edge] (x1) to [bend left] (x4);
        \draw[edge, short snake] (x4) --(x3);	
\end{tikzpicture}}}
\end{figure}
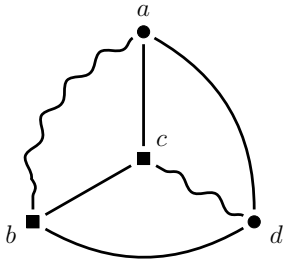

In this paper we also use the following two auxiliary results.
The next lemma~\cite[Lemma 2.3(ii)]{BoSaCoLe} allows us to decompose, under some constraints, an
edge-disjoint union of a cycle and a path into two paths.

\begin{lemma}\label{lemma:cycle+path}
	If \(G\) is a connected graph that admits a decomposition into a cycle \(C\) of length at most \(5\)
	and a path that contains at most three chords of \(C\), then \(\pn(G) = 2\).
\end{lemma}

The following lemma is a result of Fan~\cite[Lemma 3.4]{Fan05}
that is used to prove Claim~\ref{claim:even-degree-neighbors}. 

\begin{lemma}\label{lemma:fan2005}
	Let $u$ be a vertex of a graph $G$ and $G' = G - uv$, where $v \in N_G(u)$.
	Let $\D'$ be a path decomposition of $G'$.
	If \(\D'(v) > |\{w\in N_{G'}(u)\colon \D'(w) = 0\}|\),
	then there is a path decomposition \(\D\) of \(G\)
	such that \(|\D| = |\D'|\).
\end{lemma}

\section{Reducing schemes and reducing subgraphs}\label{sec:reducing}

In this section, we generalize the concept of reducing subgraph presented in~\cite{BoSaCoLe}.
This generalization combines the reducing subgraph,
which allows us to reduce the original graph by removing some of its vertices and later compensate them with the addition of not many paths,
with special paths that can either be used to extend paths in a decomposition of the reduced graph,
or to replace edges that were obtained from liftings.
Feasible Reducing schemes (see Definition~\ref{def:feasible}) and reducing subgraphs are forbidden structures in a minimal counterexample for Theorem~\ref{theo:main} (see Claim~\ref{claim:no-reducing}).

Let \(G\) be a graph, and let \(P\) be a path in \(G\) that joins two vertices \(u\) and \(v\).
We say that the edge \(uv\), denoted by \(e_P\), is \emph{parallel} to \(P\), and that $P$ is parallel to $e_P$.
Note that \(uv\) is not required to be an edge of \(G\) and also that $uv$ could be $P$ itself.
Given a set \(\mathcal{P}\) of paths in  \(G\),
we denote by \(E_\mathcal{P}\) the set \(\{e_P\colon P\in\mathcal{P}\}\)
of edges parallel to the paths in \(\mathcal{P}\).

A \emph{reducing scheme} for a graph \(G\) is a triple \(\mathcal{H}=(H,\mathcal{A},\mathcal{L})\)
where \(\{H\}\cup\mathcal{A}\cup\mathcal{L}\) is a decomposition of a subgraph \(R_\mathcal{H}\) of \(G\),
\(\mathcal{A}\cup\mathcal{L}\) is a set of paths
and \(E_\mathcal{L}\) has no edge of~\(G-E(R_\mathcal{H})\).
We denote by \(I_\mathcal{H}\) be the set of isolated vertices of~\(G-E(R_\mathcal{H})+E_\mathcal{L}\),
and we say that the graph  \(G_\mathcal{H}=G-E(R_\mathcal{H})+E_\mathcal{L}-I_\mathcal{H}\) is the \emph{reduced graph} of \(G\) through \(\mathcal{H}\).
Note that \(G_\Hcal\) is obtained from \(G\) by removing the edges of \(H\) and of the paths in \(\Acal\),
replacing the paths in \(\Lcal\) by their parallel edges,
and removing the isolated vertices.
Conversely, \(G\) can be obtained from \(G_\Hcal\) by a series of subdivisions of edges in \(E_\Lcal\), 
identifications of some of the new vertices, and by restoring \(H\) and the paths in \(\Acal\).
The following definition is the central concept of this work.

\begin{definition}\label{def:feasible}
A reducing scheme \(\mathcal{H}=(H,\mathcal{A}, \mathcal{L})\) for a graph \(G\)
is a \emph{feasible reducing scheme (FRS)} if there is a positive integer \(r\) such that
\begin{enumerate}[\rm (i)]
\item\label{def:feasible1}	\(\pn(H) \leq r\);
\item\label{def:feasible2}	\(|I_\mathcal{H}| \geq 2r\), that is, \(G_\mathcal{H}\) has at most \(|V(G)|-2r\) vertices;  
\item\label{def:feasible3}	\(\mathcal{A}\) is a set of vertex-disjoint paths such that each \(P \in\mathcal{A}\) has an end vertex \(v\) 
				with odd degree in \(G_\mathcal{H}\) and \(V(P) \setminus \{v\} \subseteq I_\mathcal{H}\);
\item\label{def:feasible4}	for each path \(P\in\mathcal{L}\) with end vertices \(u\) and \(v\),
				we have \(V(P)\setminus\{u,v\}\subseteq I_\mathcal{H}\). Moreover, for any pair of paths in 
				$\mathcal{L}$ their parallel edges are distinct;

\item\label{def:feasible5}	no vertex of \(I_\mathcal{H}\) is contained in more than one path of \(\mathcal{A}\cup\mathcal{L}\).
\end{enumerate}
In the case that \(\mathcal{H}=(H,\mathcal{A}, \mathcal{L})\) is a feasible reducing scheme with
\(\mathcal{A} = \mathcal{L}=\emptyset\), we say that~\(H\) is a \emph{reducing subgraph} (or an \emph{\(r\)-reducing subgraph}) of~\(G\).
\end{definition}

The concept of reducing subgraph presented here coincides with the the one in~\cite{BoSaCoLe},
i.e., a reducing subgraph is a subgraph with path number at most \(r\) such that the removal of its edges leaves at least \(2r\) vertices isolated.
Feasible reducing schemes extend reducing subgraphs 
by allowing some of the edges incident to vertices in \(I_\Hcal\) not to be contained in~\(H\).
These edges are covered either by the paths in \(\Acal\),
which are concatenated to paths of a path decomposition \(\D\) of \(G_\Hcal\);
or by paths in \(\Lcal\), 
which replace their parallel edges in the paths of \(\D\).
Hence, unlike reducing subgraphs, reducing schemes may explore the structure of the reduced graph.

Note that, by Definitions~\ref{def:feasible}(\ref{def:feasible4}) and~\ref{def:feasible}(\ref{def:feasible5}), 
if \(\Hcal\) is an FRS for \(G\), then \(G\) can be obtained from \(G_\Hcal\) as above,
but with no identification of vertices, i.e., by a series of subdivision of edges in \(E_\Lcal\), 
and by restoring \(H\) and the paths in \(\Acal\).

As in~\cite{BoSaCoLe}, we refer to a graph~\(G\) on \(n\) vertices as~\emph{Gallai graph} if \(\pn(G) \leq \floor{n/2}\).
The following lemma essentially proves that a minimal counterexample to Conjecture~\ref{conj:gallai} does not
admit an FRS. 

\begin{lemma}\label{lemma:reducing}
	If \(G\) is a graph that admits an FRS \(\mathcal{H} = (H,\mathcal{A},\mathcal{L})\),
	then \(\pn(G) \leq \pn(G_\mathcal{H}) + \pn(H)\).
	In particular, if \(G_\mathcal{H}\) is a Gallai graph, then \(G\) is a Gallai graph.
\end{lemma}
\begin{proof}
	Let \(G\) and \(\mathcal{H}\) be as in the statement and let \(\D\) be a minimum path decomposition of \(G_\mathcal{H}\).
	For each path \(A\in\mathcal{A}\), let \(v_A\) be the end vertex of \(A\) with odd degree in \(G_\mathcal{H}\),
	and let \(P_A\in\D\) be a path in \(\D\) with end vertex~\(v_A\)
	(if there are more than one choice for \(P_A\), we choose one arbitrarily).
    For each \(P\in\D\), we define \(\mathcal{A}_{P} = \{A\in\mathcal{A}\colon P_A = P\}\)
    and let \(\mathcal{L}_{P}\) be the set of paths in \(\mathcal{L}\) that are
	parallel to some edge of \(P\). 
	Note that \(|\Acal_P| \leq 2\), for every \(P\in\D\), and every path in \(\Acal\cup\Lcal\) is contained in \(\Acal_P\cup\Lcal_P\),
	for precisely one element \(P\) of \(\D\).
	Put \(P' = P - E_\mathcal{L} + \sum_{R\in\mathcal{L}_{P}\cup\mathcal{A}_P} R\).
	We claim that \(P'\) is a path.
        
        First note that $P'$ is connected since each edge $uv \in E_\mathcal{L}$ deleted from $P$
        is replaced by the path in $\mathcal{L}_{P}$ with end vertices $u$ and $v$.
        Now, we show that the degree of every vertex in $P'$ is at most \(2\). 
        Let $v$ be a vertex of $P$ such that $v \in V(R)$ for some
        $R\in \mathcal{A}_P \cup \mathcal{L}_{P}$. 
        By Definition~\ref{def:feasible}(\ref{def:feasible3}) and~\ref{def:feasible}(\ref{def:feasible4}), the vertex $v$ is an end vertex of $R$. 
        By Definition~\ref{def:feasible}(\ref{def:feasible3}), paths in $\mathcal{A}_P$ are vertex disjoint, 
        and by Definition~\ref{def:feasible}(\ref{def:feasible4}), paths in $\mathcal{L}_{P}$ are parallel to distinct edges.
        This implies that if $v$ is an end vertex of $P$, then there is at most one path in $\mathcal{A}_P$
        and one path in $\mathcal{L}_{P}$ with end vertex $v$, in which case the degree of $v$ in $P'$ would increase by at most \(1\) with respect to the degree in $P$, 
        and hence it is at most \(2\) (the path in $\mathcal{A}_P$ contributes with \(1\), and the path in $\mathcal{L}_{P}$ contributes with \(0\) since its parallel edge is deleted). 
        Moreover, if $v$ is an internal vertex of $P$, there are at most 
        two paths in $\mathcal{L}_{P}$ with end vertex $v$  (and no path in $\mathcal{A}_P$),
        and hence \(d_{P'}(v) = d_P(v)\). 
        Now, we consider a vertex $v$ in $V(P')\setminus V(P)$,
        i.e., $v$ is in a path $R$ of either $\mathcal{A}_P$ or $\mathcal{L}_{P}$. 
        In both cases, by Definition~\ref{def:feasible}(\ref{def:feasible5}), 
        we have \(d_{P'}(v) = d_R(v)\),
        and hence \(d_{P'}(v)\leq 2\). 
        In order to complete the proof we need to show that $P'$ is not a cycle. 
        Suppose that the end vertex $v$ of $P$ does not have degree \(1\) in~$P'$.
        It implies that there is a path $A$ in $\mathcal{A}_P$ with end vertex $v$. 
        By Definition~\ref{def:feasible}(\ref{def:feasible3}) and~\ref{def:feasible}(\ref{def:feasible5}),
        the other end vertex $u$ of $A$ has degree \(1\) in $P'$.
        
	We conclude that \(\D'=\{P'\colon P\in\D\}\) is a path decomposition of \(G-E(H)\) such that \(|\D'|=|\D|\),
	and hence, if \(\D_H\) is a minimum path decomposition of \(H\),
	then \(\D' \cup \D_H\) is a decomposition of \(G\) with cardinality \(\pn(G_\mathcal{H}) + \pn(H)\).
	Now, suppose that \(G_\mathcal{H}\) is a Gallai graph.
	By Definition~\ref{def:feasible}(\ref{def:feasible1}) and~\ref{def:feasible}(\ref{def:feasible2}),
	there is a positive integer \(r\) 
	such that \(\pn(H)\leq r\) and  \(|V(G_\mathcal{H})| \leq |V(G)|-2r\).
	Since \(G_\mathcal{H}\) is a Gallai graph,
	we have \(\pn(G_\mathcal{H}) \leq \lfloor (|V(G)|-2r)/2\rfloor = \lfloor |V(G)|/2\rfloor -r\).
	Therefore, \(\D' \cup \D_H\) is a decomposition of \(G\) with cardinality at most \(\floor{|V(G)|/2}\).
\end{proof}

Although Definition~\ref{def:feasible} and Lemma~\ref{lemma:reducing} are general,
in this paper we use mainly the case where \(\pn(H) = 1\) and \(|\mathcal{A}|,|\mathcal{L}|\leq 2\).
We stated the definition in its full generality since we believe it is a powerful tool 
to deal with Conjecture~\ref{conj:gallai}.
The next lemma is used to check planarity in the proof of Claim~\ref{claim:no-reducing}.

\begin{lemma}\label{lemma:planarity-after-lifting}
	Let \(\Hcal\) be an FRS for a graph \(G\).
	If \(G\) is planar, then \(G_\Hcal\) is planar.
\end{lemma}
\begin{proof}
Let \(G\) be as in the statement and let \(\Hcal = (H, \Acal, \Lcal)\) be an FRS for \(G\).
Note that \(G' = G - E(H) - \sum_{A \in \Acal} E(A)\) is planar,
and fix a planar drawing \(\mathcal{G}'\) of \(G'\).
Clearly, \(G_{\mathcal{H}} = G' - \sum_{P \in \Lcal} E(P) + E_{\Lcal} - I_\mathcal{H}\).
We obtain a drawing \(\mathcal{G}\) of \(G_\Hcal\) from \(\mathcal{G}'\) by drawing \(e_P \in E_\Lcal\) along the path \(P\).
By Definitions~\ref{def:feasible}(\ref{def:feasible4}) and~\ref{def:feasible}(\ref{def:feasible5}),  the paths in \(\Lcal\) are pairwise internally disjoint and, since \(\mathcal{G}'\) is a planar drawing of \(G'\), two paths in \(\Lcal\) cannot cross in an edge.
Therefore, the drawings of \(e_{P_1}\) and \(e_{P_2}\) can intersect only in its vertices, and hence \(\mathcal{G}\) is a planar drawing.
\end{proof}

\section{Main result}\label{sec:main-result}

In this section we prove the following theorem, which is the main result of this paper, 
and verifies Gallai's Conjecture for triangle-free planar graphs.
Its proof consists of a series of claims regarding the structure of a minimum counterexample.

\begin{theorem}\label{theo:main}
	Every triangle-free planar graph is a Gallai graph.
\end{theorem}

\begin{proof}
Suppose, for a contradiction, that the statement is not true, 
and let $G$ be a counterexample with a minimum number of vertices and, subject to this, with a minimum number of edges.
Note that \(G\) is connected.
In what follows let \(n = |V(G)|\) and \(m = |E(G)|\).

\begin{claim}\label{claim:no-reducing}
	\(G\) does not admit an FRS \(\mathcal{H}\) such that \(G_\mathcal{H}\) is triangle-free.
	In particular, \(G\) admits no FRS \((H, \Acal, \Lcal)\) with \(\Lcal = \emptyset\), and hence contains no reducing subgraph.
\end{claim}

\begin{proof}
	Suppose, for a contradiction, that \(\mathcal{H}\) is an FRS for $G$ in which \(G_\mathcal{H}\) a triangle-free graph.
	By Lemma~\ref{lemma:planarity-after-lifting}, \(G_\mathcal{H}\) is also planar and hence, 
	by the minimality of \(G\), the graph \(G_\mathcal{H}\) is a Gallai graph.
	Therefore, by Lemma~\ref{lemma:reducing}, \(G\) is a Gallai graph, a contradiction.
	Now, note that, if \(G\) admits an FRS \(\mathcal{H}=(H,\mathcal{A},\mathcal{L})\) with \(\mathcal{L} = \emptyset\),
	then \(G_\mathcal{H}\subset G\), and hence \(G_\mathcal{H}\) is triangle-free.
\end{proof}

We say that an edge $e \in E(G)$ is a \emph{cut-edge} if $G-e$ is not connected, and we say that cut-edge $e \in E(G)$ is \emph{useful} if each component of $G-e$ has at least two vertices.

\begin{claim}\label{claim:no-useful-cut-edge}
	\(G\) has no useful cut-edges.
\end{claim}

\begin{proof}
	Suppose, for a contradiction, that \(G\) contains a useful cut-edge \(e\), 
	and let \(G'_1\) and \(G'_2\) be the two components of \(G - e\).
	For each \(i \in \{1,2\}\), let \(G_i = G'_i+e\) and $n_i=|V(G_i)|$, and note that $n_1 +n_2 = n+2$.
        By the definition of useful cut-edge, \(G'_{i}\) contains at least two vertices.
        Thus, for each \(i \in \{1,2\}\) we have  $n_i \leq n-1$ and, by the minimality of \(G\), \(\pn(G_i) \leq \floor{n_i / 2}\).
	Let \(\D_i\) be a minimum path decomposition of~\(G_i\), and let \(P_i\in\D_i\) be the path containing the edge \(e\).
	Let \(P = P_1 + P_2\), and note that $\D = \big(\D_1 \cup \D_2 \cup \{P\}\big) \setminus \{P_1, P_2\}$ is a path decomposition of \(G\)
	such that \(|\D| \leq \lfloor n_1/2\rfloor + \lfloor n_2/2\rfloor -1 \leq \lfloor (n_1 + n_2)/2 \rfloor -1 = \lfloor n/2\rfloor\).
	Thus, \(G\) is a Gallai graph, a contradiction.
\end{proof}

A subset $S\subseteq V(G)$ is called a \emph{separator} if $G-S$ is disconnected,
and if $u,v \in V(G)$ are in distinct connected components of $G-S$, we say that $S$ \emph{separates} $u$ and $v$. 
A vertex \(v \in V(G)\) is a \emph{cut-vertex}  if \(\{v\}\) is a separator in $G$.
We say that a cut-vertex~\(v\) is \emph{useless} if~\(G-v\) has precisely two components,
and one of these components is an isolated vertex. 
Otherwise, we say that~\(v\) is \emph{useful}.
Note that a cut-vertex~\(v\) is \emph{useful} if \(G - v\) has at least three components, 
or if each of its two components has at least two vertices.

\begin{claim}\label{claim:no-useful-cut-vertex}
	\(G\) has no useful cut-vertices.
\end{claim}

\begin{proof}
	Suppose, for a contradiction, that \(G\) contains a useful cut-vertex~\(v\).
	Let \(G'_1,\ldots, G'_k\) be the components of \(G-v\) and let $n_i = |V(G'_i)|$ for each \(i\in\{1,\ldots,k\}\).
	First, suppose that~$n_j$ is even for some \(j\in\{1,\ldots,k\}\), and define \(H = G\big[V(G'_j) \cup \{v\}\big]\).
	By the minimality of~\(G\), the graph \(H\) is a Gallai graph and, since \(n_j\) is even, \(H\) has an odd number of vertices.
	Thus, \(\pn(H) \leq \floor{(n_j + 1) /2} = n_j/2\), and hence \(H\) is a $(n_j/2)$-reducing subgraph, a contradiction to Claim~\ref{claim:no-reducing}.
	Therefore, we may assume that \(n_i\) is odd for every \(i\in\{1,\ldots,k\}\).

	Now, suppose that there is \(j\in \{1,\ldots,k\}\) such that \(G'_j\)	has at least two vertices.
	Let~\(v'\) be a new vertex. We now define \(H_1 = G - V(G'_j) + vv'\) and \(H_2 = G\big[V(G'_j)\cup \{v\}\big] + vv'\).
	Note that \(n = |V(H_1)| + |V(H_2)| - 3\).	
        Since \(G'_j\)	has at least two vertices and \(G'_j \subseteq H_2\),
        we have \(|V(H_1)| = |V(G)| - |V(G'_j)| + 1 < |V(G)|\). 
        Analogously, since \(v\) is a useful cut-vertex, \(\cup_{i\neq j} V(G'_i)\) has at least two vertices,
        and hence \(|V(H_2)| < |V(G)|\). 
 	Therefore,  by the minimality of \(G\), we have that \(H_1\) and \(H_2\) are Gallai graphs, i.e., \(\pn(H_1) \leq \floor{|V(H_1)|/2}\) and \(\pn(H_2) \leq \floor{|V(H_2)|/2}\). 
 	Moreover, since \(H_2\) has an odd number of vertices,
 	\(\pn(H_2) \leq (|V(H_2)|-1)/2\).

	For \(i = 1, 2\), let \(\D_i\) be a minimum path decomposition of \(H_i\), and let \(P_i\in\D_i\) be the path containing~\(vv'\).
	Let \(P = P_1 + P_2 - vv'\), and note that \( \D = \big(\D_1 \cup D_2 \cup \{P\}\big) \setminus \{P_1, P_2\}\) is a path decomposition of $G$ such that
	\begin{eqnarray*}
	|\D| \,\, = \,\, |\D_1| + |\D_2| - 1 \quad \leq & |V(H_1)|/2 + (|V(H_2)|-1)/2 - 1 \\
					= & {\hspace*{-1cm} (|V(H_1)| + |V(H_2)| - 3)/2} \\
					= & {\hspace*{-5cm} n/2}.
	\end{eqnarray*}
	Hence \(G\) is a Gallai graph, a contradiction. 
	Therefore, we can assume that \(G'_i\) has precisely one vertex for every \(i\in\{1,\ldots,k\}\), and hence \(G\) is a star. However, every star is a Gallai graph, again a contradiction.
\end{proof}

Note that, Claim~\ref{claim:no-useful-cut-vertex} implies that if \(G\) contains a cut-vertex, then \(G\) contains a vertex of degree~\(1\). 
In Claim~\ref{claim:no-terminal-degree-1}, we verify that \(G\) has no vertices of degree \(1\),
and hence \(G\) is \(2\)-connected.

\begin{claim}\label{claim:few-small-vertices}
	\(G\) has at most one vertex with degree at most \(2\).
\end{claim}

\begin{proof}
	Suppose, for a contradiction, that \(G\) has two vertices, say \(u\) and \(v\), such that \(d(u), d(v)\leq 2\), 
	and let \(P'\) be a shortest path in \(G\) joining \(u\) and \(v\).
	Let \(S\) be the set of edges incident to \(u\) or \(v\), and put \(P = P'+S\).
	Note that \(P\) is obtained from \(P'\) by adding at most one edge at each of its end vertices.
	Thus, by the minimality of \(P'\), 
	the subgraph \(P\) is either a path or a cycle of length \(4\).
	If \(P\) is a path, then~$P$ is a 1-reducing subgraph, a contradiction to Claim~\ref{claim:no-reducing}.
	Therefore, we may assume that \(P\) is a cycle of length \(4\).
	Clearly, $G\neq P$, otherwise \(G\) would be a Gallai graph.
	By the minimality of \(G\), the graph \(G'=G\mydiff E(P)\) is a Gallai graph and hence, there is a path decomposition $\mathcal{D}$ of $G'$ of cardinality at most $\lfloor (n-2)/2\rfloor$.
	Since \(G\) is connected, there is a path $\tilde{P}$ in $\mathcal{D}$ that intersects $P$.
	By Lemma~\ref{lemma:cycle+path}, we can decompose $P+\tilde{P}$ into two paths.
	Thus, \(\pn(G) \leq (\pn(G') -1) + 2 \leq \lfloor n/2\rfloor\), and hence \(G\) is a Gallai graph, a contradiction.
\end{proof}

The next claim is essentially an application of Lemma~\ref{lemma:fan2005}.

\begin{claim}\label{claim:even-degree-neighbors}
	No vertex of \(G\) has precisely one neighbor with even degree.
\end{claim}

\begin{proof}
	Suppose, for a contradiction, that \(G\) contains a vertex \(u\) that has precisely one neighbor, say \(v\), with even degree.
	Let \(G' = G-uv\).
	By the minimality of \(G\), \(G'\) is a Gallai graph. 
	Thus, let \(\D'\) be a minimum path decomposition of \(G'\).
	Note that every vertex \(w \in N_G(u)\) has odd degree in \(G'\), and hence 
	\(\{w\in N_{G'}(u)\colon \D'(w) = 0\}=\emptyset\). 
	By Lemma~\ref{lemma:fan2005}, there is a path decomposition of \(G\) of size \(|\D'|\).
	Therefore, \(G\) is a Gallai graph, a contradiction.
\end{proof}

We say that a vertex in \(G\) is a \emph{terminal} of \(G\) if its degree is at most \(3\).

\begin{claim}\label{claim:six-terminal-vertices}
	\(G\) has at least six terminal vertices.
\end{claim}

\begin{proof}
	Since \(G\) is a triangle-free planar graph, \(m\leq 2n-4\).
	Let \(V'\) be the set of terminals of \(G\), and let \(V'' = V(G) \setminus V'\).
	Note that \(V' \neq \emptyset\), otherwise \(V'' = V(G)\) and \(4n - 8 \geq 2m = \sum_{v \in V''} d(v) \geq 4n\), a contradiction.
	Let \(w\) be a vertex of \(V'\) with minimum degree in \(G\).
	By Claim~\ref{claim:few-small-vertices}, \(d(v) = 3\) for every vertex \(v \in V' \setminus \{w\}\).
	Thus
	\begin{align*}
		2m	&= \sum_{v\in V'} d(v) + \sum_{v \in V''} d(v) \\ 
			&= d(w) + \sum_{v \in V'\setminus \{w\}} d(v) + \sum_{v \in V''} d(v)\\
			&\geq 1 + 3|V' \setminus \{w\}| + 4 |V''| \\
			&= 1 + 3\big(|V'|-1\big) + 4\big(n-|V'|\big) \\
			&= 4n - |V'| - 2.
	\end{align*}
	Therefore, \(4n - 8 \geq 2m \geq 4n - |V'| -2\), and hence \(|V'| \geq 6\).
\end{proof}

In what follows, we study the relation between the terminal vertices, following the strategy used in~\cite{BoSaCoLe}.
Claim~\ref{claim:terminals-are-not-adjacent} states that the set of terminal vertices is independent,
while Claims~\ref{claim:no-terminal-degree-1} and~\ref{claim:no-terminal-degree-2} prove
that every terminal has degree \(3\).
Claim~\ref{property:three-common-neighbors} states that two distinct terminals have at most two common neighbors,
and Claim~\ref{claim:terminal-separator} states that for every two terminals \(u\) and \(v\),
there must be two vertices \(x\) and \(y\) such that \(\{x,y\}\) separates \(u\) and \(v\).
This last property is then used to conclude the proof.

\begin{claim}\label{claim:terminals-are-not-adjacent}
	No two terminals are adjacent.
\end{claim}

\begin{proof}
	In this proof we use the following terminology. We say that two adjacent vertices \(u\) and \(v\) form a \emph{small diamond} if \(d(u) = d(v) = 3\) 
	and there exists a vertex \(w\notin N(u)\cup N(v)\) 
	such that \(\big((N(u)\cup N(v))\setminus\{u,v\}\big)\subseteq N(w)\)
	(see Figure~\ref{fig:diamond}).

	\begin{figure}[h]
		\centering
		\begin{tikzpicture}[scale = 0.8]
	\node (u) [black vertex] at (-1,1) {};
	\node (v) [black vertex] at (1,1) {};
	\node (a) [black vertex] at (-2,-1) {};
	\node (b) [black vertex] at (-1,-.75) {};
	\node (x) [black vertex] at (1,-.75) {};
	\node (y) [black vertex] at (2,-1) {};
	\node (w) [black vertex] at (0,-2) {};

	\node () [] at ($(u)+(135:.6)$)	{$u$};
	\node () [] at ($(v)+(45:.6)$)		{$v$};
	\node () [] at ($(a)+(180:.6)$)	{$b$};
	\node () [] at ($(y)+(0:.6)$)	{$y$};
	\node () [] at ($(b)+(45:.6)$)	{$a$};
	\node () [] at ($(x)+(135:.6)$)	{$x$};
	\node () [] at ($(w)+(-90:.6)$)	{$w$};
        \draw[edge] (u) -- (v);
        \draw[edge] (u) -- (a);
        \draw[edge] (u) -- (b);
        \draw[edge] (v) -- (x);
        \draw[edge] (v) -- (y);
        \draw[edge] (a) -- (w);
        \draw[edge] (b) -- (w);
        \draw[edge] (x) -- (w);
        \draw[edge] (y) -- (w);
\end{tikzpicture}
		\caption{Example of a small diamond. }\label{fig:diamond}
	\end{figure}
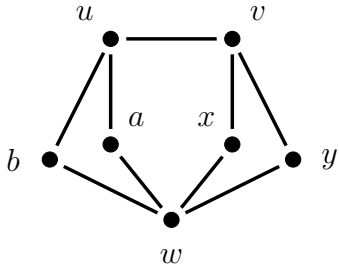

	\begin{subclaim}\label{claim:adjacent-terminals-1}
		Every two adjacent terminals form a small diamond.
	\end{subclaim}

	\begin{proof}[Proof of Subclaim~\ref{claim:adjacent-terminals-1}]
	    Let \(u\) and \(v\) be two adjacent terminals.
		Suppose, without loss of generality, that \(d(v) \leq d(u)\).
		By Claim~\ref{claim:few-small-vertices}, we have \(d(u) = 3\), and hence let \(N(u) = \{v, a, b\}\).
		First, suppose that \(d(v) \leq 2\).
		Let~\(S\) be the set of edges incident to \(v\).
		Note that \(d(a)\) is odd, otherwise \((bu + S, \{au\},\emptyset)\) 
		would be an FRS for \(G\) (see Figure~\ref{fig:m4}), 
		a contradiction to Claim~\ref{claim:no-reducing}.
		By symmetry, \(d(b)\) is odd as well.
		Therefore, by Claim~\ref{claim:even-degree-neighbors}, \(d(v)\) is odd, and hence \(d(v) = 1\).
		By Claim~\ref{claim:few-small-vertices}, \(d(a) \geq 3\), and by Claim~\ref{claim:no-useful-cut-edge},
			\(ua\) is not a cut-edge.
		Let~\(P'\) be a path joining \(a\) and \(b\) in \(G - ua\), and note that \(u, v \notin V(P')\).
		Hence, \((P' + buv, \{au\},\emptyset)\) is an FRS for \(G\) (see Figure~\ref{fig:m5}), 
		a contradiction to Claim~\ref{claim:no-reducing}.

		\begin{figure}[h]
			\centering
			\begin{subfigure}[b]{.33\linewidth}
				\centering\scalebox{.7}{\begin{tikzpicture}[scale = 0.6]
	\node (u) [squared black vertex] at ($(0:0)$) {};
	\node (a) [black vertex] at ($(u)+(-45:3)$) {};
	\node (b) [black vertex] at ($(u)+(-135:3)$) {};

	\node (label_u) [] at ($(u)+(90:.6)$) {$u$};
	\node (label_a) [] at ($(a)+(-90:.6)$) {$a$};
	\node (label_b) [] at ($(b)+(-90:.6)$) {$b$};

	\begin{scope}[xshift=5cm]
	\node (v) [squared black vertex] at ($(0:0)$) {};
	\node (z) [white vertex] at ($(v)+(-45:3)$) {};

	\node (label_v) [] at ($(v)+(90:.6)$) {$v$};
	\node (label_z) [] at ($(z)+(-90:.6)$) {$z$};
	\end{scope}

	\draw[edge, color1] (b) -- (u) -- (v); 
	\draw[edge, color2] (a) -- (u); 
	\draw[possible edge, color1] (v) -- (z); 

    \node (label_Pp) [] at ($(u)+(-90:2.8)$) {\phantom{$P'$}};
\end{tikzpicture}}
				\caption{}\label{fig:m4}
			\end{subfigure}%
			\hfill
			\begin{subfigure}[b]{.33\linewidth}
				\centering\scalebox{.7}{\begin{tikzpicture}[scale = 0.6]
	\node (u) [squared black vertex] at ($(0:0)$) {};
	\node (a) [black vertex] at ($(u)+(-45:3)$) {};
	\node (b) [black vertex] at ($(u)+(-135:3)$) {};

	\node (label_u) [] at ($(u)+(90:.6)$) {$u$};
	\node (label_a) [] at ($(a)+(-90:.6)$) {$a$};
	\node (label_b) [] at ($(b)+(-90:.6)$) {$b$};

    \node (label_Pp) [] at ($(u)+(-90:2.8)$) {$P'$};

	\begin{scope}[xshift=5cm]
	\node (v) [squared black vertex] at ($(0:0)$) {};

	\node (label_v) [] at ($(v)+(90:.6)$) {$v$};
	\end{scope}

	\draw[edge, color1] (b) -- (u) -- (v); 
	\draw[edge, short snake, color1] (b) to (a);
	\draw[edge, color2] (a) -- (u); 
\end{tikzpicture}}
				\caption{}\label{fig:m5}
			\end{subfigure}%
			\hfill
			\begin{subfigure}[b]{.33\linewidth}
				\centering\scalebox{.7}{\begin{tikzpicture}[scale = 0.6]
	\node (u) [squared black vertex] at ($(0:0)$) {};
	\node (a) [black vertex] at ($(u)+(-45:3)$) {};
	\node (b) [black vertex] at ($(u)+(-135:3)$) {};

	\node (label_u) [] at ($(u)+(90:.6)$) {$u$};
	\node (label_a) [] at ($(a)+(-90:.6)$) {$a$};
	\node (label_b) [] at ($(b)+(-90:.6)$) {$b$};

	\begin{scope}[xshift=7cm]
	\node (v) [squared black vertex] at ($(0:0)$) {};
	\node (x) [black vertex] at ($(v)+(-135:3)$) {};
	\node (y) [black vertex] at ($(v)+(-45:3)$) {};

	\node (label_v) [] at ($(v)+(90:.6)$) {$v$};
	\node (label_x) [] at ($(x)+(-90:.6)$) {$x$};
	\node (label_y) [] at ($(y)+(-90:.6)$) {$y$};
	\end{scope}

	\draw[edge, color2] (b) -- (u) -- (v); 
	\draw[edge, color1]  (u) -- (a) (v) -- (x) (v) -- (y);
	\draw[edge, short snake, color1] (a) to   (x);

    \node (label_Pp) [] at ($(u)+(0:3.5)+(-90:2.8)$) {$P'$};
\end{tikzpicture}}
				\caption{}\label{fig:m6}
			\end{subfigure}%
			\caption{Each figure illustrates a reducing scheme $(H,\Acal,\Lcal)$ where $H$ is illustrated in blue, 
				and the paths in $\mathcal{A}$ are illustrated in red.
				In (a) we used a dashed line and a white circle to denote the neighbor of \(v\) 
				that does not necessarily exists.
				}
			\label{fig:case1}
		\end{figure}
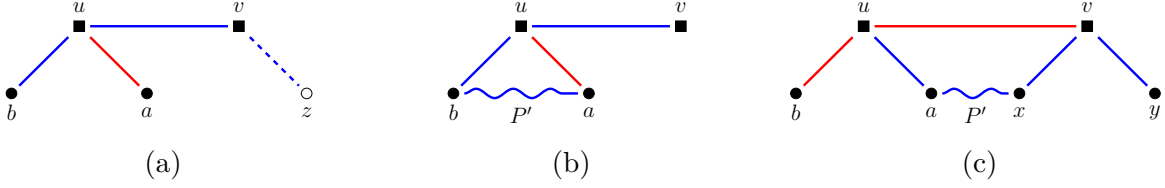

		Therefore, we may suppose \(d(v) = 3\), and  let \(N(v) = \{u, x, y\}\).
		By Claim~\ref{claim:no-useful-cut-edge},
		since \(d(u) = d(v) = 3\),~\(uv\) is not a cut-edge.
		Let \(P'\) be a shortest path in \(G-uv\) joining \(u\) and \(v\), and suppose without loss of generality that \(a,x\in V(P')\).
		By the minimality of \(P'\), \(b,y\notin V(P')\).
		Thus \(d(b)\) is odd, otherwise \((P' + vy, \{buv\},\emptyset)\) would be an FRS for~\(G\)
		(see Figure~\ref{fig:m6}),
		a contradiction to Claim~\ref{claim:no-reducing}.
		By symmetry, \(d(y)\) is also odd and, by Claim~\ref{claim:even-degree-neighbors}, \(d(a)\) and \(d(x)\) are odd.
		We claim that \(dist_{G-u-v}(b,y) \leq 2\).
		Suppose, for a contradiction, that \(dist_{G-u-v}(b,y) \geq 3\), and let \(G' = G -u-v+by\).
		Since \(dist_{G - u - v} (b, y) \geq 3\), \(G'\) is triangle-free.
		However, \(\Hcal=(P',\emptyset,\{buvy\})\) is an FRS for \(G\) such that \(G_\Hcal=G'\) (see Figure~\ref{fig:m7}),
		a contradiction to Claim~\ref{claim:no-reducing}.
		Thus, we have \(dist_{G-u-v}(b,y) \leq 2\).
		 
		Note that for every path \(Q'\) joining \(b\) and \(y\) in \(G - u - v\) 
		we have \(V(Q') \cap V(P') \neq \emptyset\), 
		otherwise \((P' + vy + Q', \{buv\},\emptyset)\) would be an FRS for \(G\) (see Figure~\ref{fig:m8}),
		a contradiction to Claim~\ref{claim:no-reducing}.
		Let~\(Q'\) be a shortest path joining \(b\) and \(y\) in \(G - u - v\),
		and let \(w\in V(Q')\cap V(P')\).
		Since \(G\) is triangle-free, \(a, x \notin V(Q')\).
		Also, by the minimality of \(P'\), 
		we have \(dist_{G - u - v} (a, x) \leq dist_{G-u-v}(b,y)\).
		Thus, $dist_{G-u-v} (a,x) = dist_{G-u-v} (b,y) = 2$,
		and hence \(\{a,b,x,y\}\subseteq N(w)\).
		Therefore, \(u\) and \(v\) form a small diamond.
	\end{proof}

	\begin{figure}[h]
		\centering
		\begin{subfigure}[b]{.45\linewidth}
			\centering\scalebox{.7}{\begin{tikzpicture}[scale = 0.6]
	\node (u) [squared black vertex] at ($(0:0)$) {};
	\node (a) [black vertex] at ($(u)+(-45:3)$) {};
	\node (b) [black vertex] at ($(u)+(-135:3)$) {};

	\node (label_u) [] at ($(u)+(90:.6)$) {$u$};
	\node (label_a) [] at ($(a)+(180:.6)$) {$a$};
	\node (label_b) [] at ($(b)+(180:.6)$) {$b$};

	\node (anchor_b) [] at ($(b)+(-45:1.5)$) {};

        \node (label_path2) [] at ($(u)+(0:3.5) + (-90:2.8)$) {$P'$};

	\begin{scope}[xshift=7cm]
	\node (v) [squared black vertex] at ($(9:0)$) {};
	\node (x) [black vertex] at ($(v)+(-135:3)$) {};
	\node (y) [black vertex] at ($(v)+(-45:3)$) {};

	\node (anchor_y) [] at ($(y)+(-135:1.5)$) {};

	\node (label_v) [] at ($(v)+(90:.6)$) {$v$};
	\node (label_x) [] at ($(x)+(0:.6)$) {$x$};
	\node (label_y) [] at ($(y)+(0:.6)$) {$y$};
	\end{scope}

	\draw[edge, color3] (b) -- (u) -- (v)  (v) -- (y);
	\draw[edge, color1]  (u) -- (a) (v) -- (x);
	\draw[edge, short snake, color1] (a) to  (x);
	
	\draw[edge, short snake, color=white] (b) to [bend right= 30]  (y);

\end{tikzpicture}}
			\caption{}\label{fig:m7}
		\end{subfigure}%
		\hfill
		\begin{subfigure}[b]{.45\linewidth}
			\centering\scalebox{.7}{\begin{tikzpicture}[scale = 0.6]
	\node (u) [squared black vertex] at ($(0:0)$) {};
	\node (a) [black vertex] at ($(u)+(-45:3)$) {};
	\node (b) [black vertex] at ($(u)+(-135:3)$) {};

	\node (label_u) [] at ($(u)+(90:.6)$) {$u$};
	\node (label_a) [] at ($(a)+(180:.6)$) {$a$};
	\node (label_b) [] at ($(b)+(180:.6)$) {$b$};

        \node (label_path1) [] at ($(u) + (0:3.5) + (-90:1.6)$) {$P'$};
        \node (label_path2) [] at ($(u) + (0:3.5) + (-90:3.4)$) {$Q'$};

	\begin{scope}[xshift=7cm]
	\node (v) [squared black vertex] at ($(0:0)$) {};
	\node (x) [black vertex] at ($(v)+(-135:3)$) {};
	\node (y) [black vertex] at ($(v)+(-45:3)$) {};

	\node (label_v) [] at ($(v)+(90:.6)$) {$v$};
	\node (label_x) [] at ($(x)+(0:.6)$) {$x$};
	\node (label_y) [] at ($(y)+(0:.6)$) {$y$};
	\end{scope}

	\draw[edge, color2] (b) -- (u) -- (v);
	\draw[edge, color1]  (u) -- (a) (v) -- (x) (v) -- (y);
	\draw[edge, short snake, color1] (a) to  (x);
	\draw[edge, short snake, color1] (b) to [bend right= 30]  (y);
\end{tikzpicture}}
			\caption{}\label{fig:m8}
		\end{subfigure}%
		\caption{Each figure illustrates a reducing scheme $(H,\Acal,\Lcal)$ where $H$ is illustrated in blue, 
				the paths in $\mathcal{A}$ are illustrated in red,
				and the paths in $\Lcal$ are illustrated in green.}
		\label{fig:case}
	\end{figure}
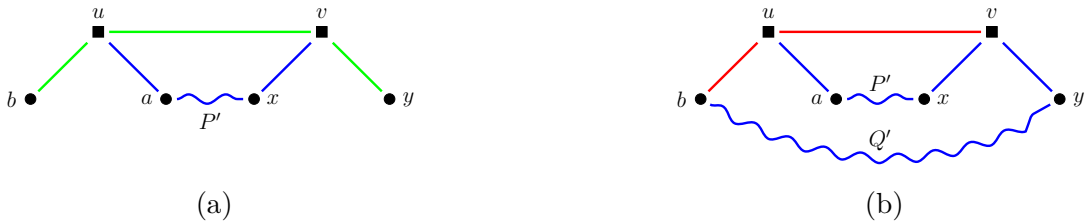

	\begin{subclaim}\label{claim:adjacent-terminals-2}
		Every terminal has at most one neighbor that is a terminal.
	\end{subclaim}

	\begin{proof}[Proof of Subclaim~\ref{claim:adjacent-terminals-2}]
		Suppose that \(u\) and \(v\) are adjacent terminals.
		By Subclaim~\ref{claim:adjacent-terminals-1}, \(u\) and \(v\) form a small diamond.
		Let \(N(u) = \{v, a, b\}\),  \(N(v) = \{u, x, y\}\), and let \(w\) be  a vertex such that \(\{a,b, x, y\} \subseteq N(w)\).
		Suppose, without loss of generality, that \(a\) is a terminal.
		By Subclaim~\ref{claim:adjacent-terminals-1} applied to \(a\) and \(u\), there is a vertex \(w'\) such that \(w,b\in N(w')\) because \(w\in N(a)\setminus\{u\}\) and \(b\in N(u)\setminus\{a\}\).
		But since \(bw\in E(G)\), there is a triangle in~\(G\), a contradiction.
	\end{proof}

	As before let \(u\) and \(v\) be adjacent terminals.
	Let \(N(u) = \{v,a,b\}\) and \(N(v) = \{u, x,y\}\), and let \(w\) be the vertex given by Claim~\ref{claim:adjacent-terminals-1} (see Figure~\ref{fig:diamond}).
	We claim that \(dist_{G-u}(a,v) \geq 3\).
	Indeed, suppose that there is a path \(Q\) of length at most~\(2\) in \(G-u\) joining~\(a\) and~\(v\).
	Since \(a \notin N(v)\), \(Q\) has length \(2\). 
	In this case, \(a\in N(x)\cup N(y)\), but then either \(awxa\) or \(awya\) is a triangle in \(G\), a contradiction.
	By symmetry, we have \(dist_{G-u}(b,v), dist_{G-v}(x,u), dist_{G-v}(y,u)  \geq 3\).

    Now, let \(z\) be a terminal of \(G\) different from \(u\) and \(v\).
    Note that \(z\) exists by Claim~\ref{claim:six-terminal-vertices}. 
    Moreover, by Claim~\ref{claim:adjacent-terminals-2}, \(z\neq a,b,x,y\), and since \(d(z)\leq 3\) and \(d(w)\geq 4\), we have \(z\neq w\).
	We claim that \(d(z) = 3\).
	Indeed, suppose that \(d(z)\leq2\), and let \(P'\) be a shortest path joining \(z\) to a vertex \(u'\) in \(\{a,b,x,y\}\).
	Suppose, without loss of generality, that \(u'=a\)
	and let \(\{z'\} = N(z) \cap V(P')\).
	If \(d(z) = 1\), then put \(P = ua + P'\). 
	In the case that \(d(z) = 2\), let \(x'\in N(z) \setminus \{z'\}\) and put \(P = P' + zx' + au\).
	Let \(G' = G -u-z-E(P) + bv\).
	Since \(dist_{G-u}(b,v)\geq 3\), \(G'\) is triangle-free.
	However, \(\Hcal=(P,\emptyset,\{buv\})\) is an FRS for \(G\) such that \(G_\Hcal=G'\),
	a contradiction to Claim~\ref{claim:no-reducing}.
	Thus, we may assume that \(d(z) = 3\) for every terminal \(z\) of \(G\).
	Note that this implies, by Claim~\ref{claim:no-useful-cut-vertex}, that \(G\) has no cut-vertex.
	In what follows, fix a terminal \(z\) different from \(u\) and \(v\).

	\begin{subclaim}\label{subclaim:special-path}
		There is a path \(P\) joining a vertex \(u'\) in \(\{a,b,x,y\}\) to a vertex \(z'\) in \(N(z)\) such that:
		\begin{enumerate}[\rm 1.]
			\item\label{it:one} 	\(P\) contains no vertex from \(\{a,b,x,y\} \setminus \{u'\}\) and no vertex from \(N(z) \setminus \{z'\}\); and
			\item\label{it:two} 	Let $\mathcal{P}^z_{z'}$ be the set of paths of length \(2\) in \(G-z\) joining the two vertices in \(N(z) \setminus \{z'\}\). 
						If $\mathcal{P}^z_{z'}\neq \emptyset$, then there is 
						\(Q\in \mathcal{P}^z_{z'}\) such that \(V(Q)\cap V(P)=\emptyset\).
		\end{enumerate}
	\end{subclaim}

	\begin{proof}[Proof of Subclaim~\ref{subclaim:special-path}]
	    Let \(N(z) = \{x',y',z'\}\),
	    and let \(P\) be a shortest path from a vertex in \(\{a,b,x,y\}\) to a vertex in \(N(z)\).
		Then \(P\) contains at most one vertex of \(\{a,b,x,y\}\), and at most one vertex of \(N(z)\). 
		Hence, $P$ satisfies~\ref{it:one}.
		Suppose, without loss of generality, that \(V(P)\cap N(z) = \{z'\}\).
		If \(\mathcal{P}^z_{z'}=\emptyset\), then \(P\) satisfies \ref{it:two}.
		Thus, we may assume that \(\mathcal{P}^z_{z'}\neq\emptyset\), and let \(Q=x'w'y'\in\mathcal{P}^z_{z'}\).
		Assume that \(V(Q)\cap V(P)\neq \emptyset\), otherwise \(P\) also satisfies~\ref{it:two}. 
		Clearly, we have \(V(Q)\cap V(P)= \{w'\}\).
		Moreover, by the minimality of \(P\), the vertex \(w'\) is a neighbor of \(z'\) in \(G\), 
		and hence \(\{x', y', z'\} \subseteq N(w')\). 
		Since \(G - w'\) is connected, 
		there is a shortest path \(P'\) in \(G - w'\) joining a vertex in \(\{a,b,x,y\}\) to a vertex \(z''\) in \(N(z)\).
		Moreover, by the minimality of \(P'\), we have \(V(P') \cap N(z) = \{z''\}\).
		Hence \(P'\) satisfies~\ref{it:one}.
		Let \(\{v_1, v_2\} = N(z) \setminus \{z''\}\). 
		Then \(Q' = v_1w'v_2\) is a path of length \(2\) in \(G - z\) 
		joining the two vertices of \(N(z) \setminus \{z''\}\) such that \(V(Q') \cap V(P') = \emptyset\).
		Thus, $P'$ also satisfies~\ref{it:two}.
		\end{proof}

	Now, let \(P'\) be a path given by Subclaim~\ref{subclaim:special-path}.
	Let \(N(z) = \{x',y',z'\}\) and suppose, without loss of generality, that \(a, z' \in V(P')\) and \(y', x', b, x, y \notin V(P')\).
	Suppose that \(y'\) has even degree in $G$.
	Let \(G' = G - u - z - E(P') + bv\).
	Since \(dist_{G-u}(b,v)\geq 3\), \(G'\) is triangle-free. 
	Let \(P = x'zz' + P' + au\).
	Since \(u\) and \(z\) are not adjacent, we have \(x'\neq u\), and hence \(P\) is a path.
	However, \(\Hcal=(P,\{y'z\},\{buv\})\) is an FRS for \(G\) such that \(G_\Hcal=G'\),	
        a contradiction to Claim~\ref{claim:no-reducing}.
	Thus, \(d(y')\) is odd and, by symmetry \(d(x')\) is odd as well.
	Moreover, by Claim~\ref{claim:even-degree-neighbors}, \(z'\) has odd degree. 
	Hence, \(x'\), \(y'\), and \(z'\) have odd degree in \(G\).
	
	In what follows we divide the proof on whether \(\mathcal{P}^z_{z'}=\emptyset\) or \(\mathcal{P}^z_{z'}\neq\emptyset\).
	First, suppose that $\mathcal{P}^z_{z'}=\emptyset$.
	Let \(G' = G -u - z - E(P') + x'y' + bv\).
	Since~$\mathcal{P}^z_{z'}=\emptyset$ and \(dist_{G - u}(b, v) \geq 3\), 
	we have that \(G'\) has no triangle containing exactly one of the edges \(x'y'\), \(bv\).
	Thus, if \(G'\) has a triangle, then it contains both edges \(x'y'\) and \(bv\),
	and hence the edges \(x'y'\) and \(bv\) have a vertex in common.
	Since \(u\) and \(v\) are adjacent,
	Claim~\ref{claim:adjacent-terminals-2} implies that \(y', x' \neq v\).
	Thus, \(b = y'\) or \(b = x'\).
	Suppose, without loss of generality, that \(b = y'\).
	Then, the edge \(x'v\) is necessarily in \(G\).
	Suppose, without loss of generality, that \(x' = x\).
	Therefore, \(w\) is adjacent to both \(b=y'\) and \(x=x'\),
	a contradiction to our assumption $\mathcal{P}^z_{z'}=\emptyset$.
	Thus, \(G'\) is triangle-free.
	However, \(\Hcal=(au + P' + z'z,\emptyset,\{x'zy',buv\})\) is an FRS for \(G\) such that \(G_\Hcal=G'\),	
    a contradiction to Claim~\ref{claim:no-reducing}.
	
	Now, suppose that~$\mathcal{P}^z_{z'}\neq \emptyset$. 
	By Subclaim~\ref{subclaim:special-path}, there
	is a path \(Q' = y'w'x'\) of length~\(2\) in \(G - z\) such that \(V(Q') \cap V(P') = \emptyset\).
	Let \(G' = G -u - z - E(P') -E(Q') + bv\).
	Since \(dist_{G - u}(b, v) \geq 3\), \(G'\) has no triangle.
	Since \(x'\), \(y'\), and \(z'\) have odd degree in \(G\), 
	and \(d_{G'}(y') = d_G(y')-2\), the vertex \(y'\) has odd degree in \(G'\).
	Let \(P = au + P' +  z'zx'w'y'\).
	If \(w'=u\), then \(\{y',x'\}=\{b,v\}\), hence \(y'zx'\) is a path of length \(2\) joining \(b\) to \(v\), 
	a contradiction to \(dist_{G - u}(b, v) \geq 3\). Thus, \(P\) is a path.
	However, \(\Hcal=(P,\{zy'\},\{buv\})\) is an FRS for \(G\) such that \(G_\Hcal=G'\),	
    a contradiction to Claim~\ref{claim:no-reducing}.
	This finishes the proof of Claim~\ref{claim:terminals-are-not-adjacent}.
\end{proof}

\begin{claim}\label{claim:no-terminal-degree-1}
	No terminal has degree \(1\).
\end{claim}

\begin{proof}
	Suppose, for a contradiction, that \(G\) contains a terminal~\(u\) with degree~\(1\), and let~\(u'\) be the only neighbor of \(u\).
	By Claim~\ref{claim:few-small-vertices}, we may assume that \(d(v) = 3\) for every terminal~\(v\) of~\(G\) different from \(u\).
	We say that \(u\) and a terminal \(v \in V(G) \setminus \{u\}\) form a \emph{\((u,v)\)-kite} if \(N(v) \subseteq N(u')\).
	The next subclaim is similar to Subclaim~\ref{subclaim:special-path}.

	\begin{subclaim}\label{claim:path-satisfing-conditions}
		If \(v \in V(G) \setminus \{u\}\) is a terminal of \(G\) and \(u\) and \(v\) do not form a \((u,v)\)-kite, then there is a path \(P\) joining \(u\) and \(v\) such that:
		\begin{enumerate}[\rm 1.]
			\item\label{it:psc-1} \(P\) contains precisely one vertex, say \(v'\), in \(N(v)\); and
			\item \label{it:psc-2} Let $\mathcal{P}^{u,v}_{v'}$ be the set of paths of length \(2\) in \(G-u-v\) joining the two vertices in \(N(v) \setminus \{v'\}\). 
			If $\mathcal{P}^{u,v}_{v'} \neq \emptyset$, then there is a path \(Q\in\mathcal{P}^{u,v}_{v'}\) such that \(V(Q)\cap V(P)=\emptyset\).
		\end{enumerate}
	\end{subclaim}
 	
	\begin{proof}[Proof of Subclaim~\ref{claim:path-satisfing-conditions}]
		Let \(N(v) = \{x, y, z\}\) and let \(P\) be a shortest path joining \(u\) and~\(v\).
		Note that \(u' \in V(P)\) and that \(P\) contains precisely one vertex of \(N(v)\). 
		Thus, $P$ satisfies~\ref{it:psc-1}.
		Suppose, without loss generality, that \(V(P)\cap N(v) = \{x\}\).
		We may assume that \(\mathcal{P}^{u,v}_{x}\neq\emptyset\), otherwise \(P\) satisfies \ref{it:psc-2}.
		Thus, let \(Q=ywz\in\mathcal{P}^{u,v}_{x}\).	
		If \(w \notin V(P)\), then \(V(Q)\cap V(P)=\emptyset\), and the result follows.
		Then, suppose that \(w \in V(P)\).
		By the minimality of \(P\), the vertex \(w\) must be a neighbor of \(x\) in \(G\), and hence \(\{x,y,z\} \subseteq N(w)\).
		Note that \(w\) must be different of \(u'\), otherwise \(u\) and \(v\) would form a \((u, v)\)-kite.

		Now, suppose that \(w\) is a cut-vertex of~\(G\).
		By Claim~\ref{claim:no-useful-cut-vertex}, \(w\) is not useful, and hence \(G - w\) 
		has a component which consists of an isolated vertex. However, this isolated vertex has degree \(1\) in $G$  and it is distinct from $u$, 
		a contradiction to Claim~\ref{claim:few-small-vertices}.				
		Therefore, we may assume that \(G - w\) is connected, and hence there is a shortest path \(R\) in \(G-w\) joining~\(u\) to~\(v\).
		By the minimality of~\(R\), we have that \(R\) contains precisely one vertex, say \(v'\), in \(N(v)\). 
		Moreover, let \(\{v_1,v_2\} = N(v)\setminus\{v'\}\), and note that \(Q' = v_1wv_2\in\mathcal{P}^{u,v}_{v'}\) is 
		such that \(V(Q')\cap V(R)= \emptyset\).
		Thus, \(R\) satisfies~\ref{it:psc-1} and~\ref{it:psc-2}.
	\end{proof}

	\begin{subclaim}\label{claim:degree-1-kite}
		For every terminal \(v\) of \(G\) different from \(u\), the vertices \(u\) and \(v\) form a \((u,v)\)-kite and every vertex in \(N(v)\) has odd degree.
	\end{subclaim}

	\begin{proof}[Proof of Subclaim~\ref{claim:degree-1-kite}] 
		Let \(v \in V(G) \setminus \{u\}\) be a terminal and let \(N(v) = \{x, y, z\}\).
		First, we prove that every vertex in \(N(v)\) has odd degree.
		Suppose, for a contradiction, that \(N(v)\) contains a vertex with even degree.
		Let \(P'\) be a shortest path joining \(u\) and \(v\), and suppose, without loss of generality, that \(P'\) contains~\(x\) (hence \(y, z \notin V(P')\)).
		By Claim~\ref{claim:even-degree-neighbors}, at least two vertices in \(N(v)\) have even degree, and hence at least
		one vertex between \(y\) and \(z\), say \(y\), has even degree.
		Therefore, \((P' + vz, \{yv\},\emptyset)\) is an FRS (see Figure~\ref{fig:m9}), a contradiction to Claim~\ref{claim:no-reducing}.

		Now we prove that \(u\) and \(v\) form a \((u, v)\)-kite.
		Suppose, for a contradiction, that \(u\) and~\(v\) do not form a \((u, v)\)-kite.
                Let $P'$ be a path given by Subclaim~\ref{claim:path-satisfing-conditions}. 
		Suppose, without loss of generality, that \(x \in V(P')\) (hence \(y, z \notin V(P')\)).
		First, suppose that \(\mathcal{P}^{u,v}_{x}=\emptyset\), and let \(G' = G -u -v -E(P') + yz\).
		Note that \(G'\) is triangle-free.
		However, \(\Hcal=(P',\emptyset,\{yvz\})\) is an FRS for \(G\) such that \(G_\Hcal=G'\) (see Figure~\ref{fig:m10}),	
		a contradiction to Claim~\ref{claim:no-reducing}.
		Now, suppose that~\(\mathcal{P}^{u,v}_{x}\neq\emptyset\) and let $Q'\in\mathcal{P}^{u,v}_{x}$ be a path such that \(V(Q') \cap V(P')=\emptyset\).
		Then, \((P' + zv + Q', \{yv\},\emptyset)\) is an FRS for \(G\) (see Figure~\ref{fig:m11}), 
		a contradiction to Claim~\ref{claim:no-reducing}.
	\end{proof}

	\begin{figure}[h]
		\centering
		\begin{subfigure}[b]{.33\linewidth}
			\centering\scalebox{.7}{\begin{tikzpicture}[scale = 0.6]

	\node (u) [squared black vertex] at ($(0:0)$) {};

	\node (label_u) [] at ($(u)+(180:.6)$) {$u$};

	\node (v) [squared black vertex] at ($(u)+(0:6)$) {};
    \node (y) [black vertex] at ($(v)+(-180:2)$) {};
    \node (x) [black vertex] at ($(y)+(90:1.8)$) {};
	\node (z) [black vertex] at ($(y)+(-90:1.8)$) {};

	\node (label_v) [] at ($(v)+(90:.6)$) {$v$};
	\node (label_x) [] at ($(x)+(90:.6)$) {$x$};
	\node (label_y) [] at ($(y)+(90:.6)$) {$y$};
	\node (label_z) [] at ($(z)+(-90:.6)$) {$z$};

	\node (anchor_y) [] at ($(y)+(170:1.5)$) {};
	\node (anchor_z) [] at ($(z)+(-170:1.5)$) {};

        \node (label_path1) [] at ($(u)+(0:1.7)+(0,1.8)$) {$P'$};

	\draw[edge, color1] (x) -- (v) (v) -- (z);
	\draw[edge, short snake, color1] (u) to [bend left = 10] (x);
	\draw[edge,  color2] (v) to  (y) ;
\end{tikzpicture}}
			\caption{}\label{fig:m9}
		\end{subfigure}%
		\hfill
		\begin{subfigure}[b]{.33\linewidth}
			\centering\scalebox{.7}{\begin{tikzpicture}[scale = 0.6]

	\node (u) [squared black vertex] at ($(0:0)$) {};
	\node (a) [black vertex] at ($(u)+(0:2.17)$) {};

	\node (label_u) [] at ($(u)+(180:.6)$) {$u$};
	\node (label_a) [] at ($(a)+(90:.6)$) {$u'$};

	\node (v) [squared black vertex] at ($(u)+(0:9)$) {};
    \node (y) [black vertex] at ($(v)+(-180:2)$) {};
    \node (x) [black vertex] at ($(y)+(90:1.8)$) {};
	\node (z) [black vertex] at ($(y)+(-90:1.8)$) {};

	\node (label_v) [] at ($(v)+(90:.6)$) {$v$};
	\node (label_x) [] at ($(x)+(90:.6)$) {$x$};
	\node (label_y) [] at ($(y)+(90:.6)$) {$y$};
	\node (label_z) [] at ($(z)+(-90:.6)$) {$z$};

	\node (anchor_y) [] at ($(y)+(170:1.5)$) {};
	\node (anchor_z) [] at ($(z)+(-170:1.5)$) {};

        \node (label_path1) [] at ($(u)+(0:4)+(0,2.5)$) {$P'$};

	\draw[edge, color1] (u) -- (a) (x) -- (v) ;
	\draw[edge, short snake, color1] (a) to [bend left = 30] (x);
	\draw[edge,  color3] (v) to  (y) (v) -- (z);
\end{tikzpicture}}
			\caption{}\label{fig:m10}
		\end{subfigure}%
		\hfill
		\begin{subfigure}[b]{.33\linewidth}
			\centering\scalebox{.7}{\begin{tikzpicture}[scale = 0.6]

	\node (u) [squared black vertex] at ($(0:0)$) {};
	\node (a) [black vertex] at ($(u)+(0:2.17)$) {};

	\node (label_u) [] at ($(u)+(180:.6)$) {$u$};
	\node (label_a) [] at ($(a)+(90:.6)$) {$u'$};

	\node (v) [squared black vertex] at ($(u)+(0:9)$) {};
    \node (y) [black vertex] at ($(v)+(-180:2)$) {};
    \node (x) [black vertex] at ($(y)+(90:1.8)$) {};
	\node (z) [black vertex] at ($(y)+(-90:1.8)$) {};
	\node (w) [black vertex] at ($(y)+(-90:.9)+(180:2)$) {};

	\node (label_v) [] at ($(v)+(90:.6)$) {$v$};
	\node (label_x) [] at ($(x)+(0:.6)$) {$x$};
	\node (label_y) [] at ($(y)+(90:.6)$) {$y$};
	\node (label_z) [] at ($(z)+(-90:.6)$) {$z$};

        \node (label_path1) [] at ($(u)+(0:4)+(0,2.5)$) {$P'$};
        \node (label_path2) [] at ($(w)+(50:1.5)$) {$Q'$};

	\draw[edge, color1] (u) -- (a) (x) -- (v) (v) -- (z) (z) -- (w) (w) -- (y);
	\draw[edge, short snake, color1] (a) to [bend left = 30] (x);
	\draw[edge,  color2] (v) to  (y);
\end{tikzpicture}}
			\caption{}\label{fig:m11}
		\end{subfigure}%
		\caption{Illustrations for the proof of Subclaim~\ref{claim:degree-1-kite}.
		        Each figure illustrates a reducing scheme $(H,\Acal,\Lcal)$ where $H$ is illustrated in blue, 
				the paths in $\mathcal{A}$ are illustrated in red,
				and the paths in $\Lcal$ are illustrated in green.}
		\label{fig:subclaim81}
	\end{figure}
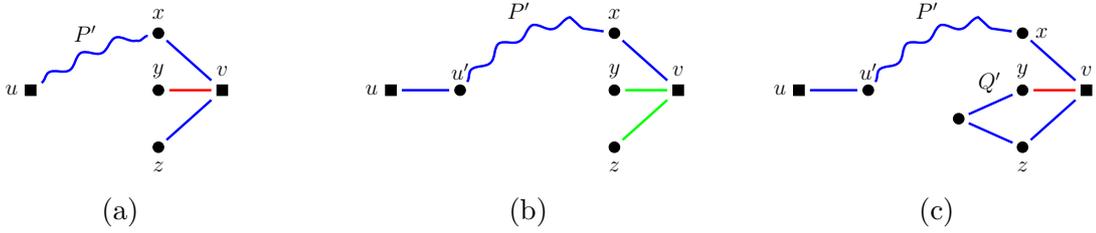

	Now, we complete the proof of Claim~\ref{claim:path-satisfing-conditions}.
	For that, fix a planar drawing of \(G\).
	Given a terminal \(v \in V(G) \setminus \{u\}\), we have \(N(v)\subseteq N(u')\).
	For every two neighbors, say \(x,y\) of \(v\), the cycle \(vxu'yv\) is a square in \(G\).
	Thus, there are three squares containing \(v\) and \(u'\), say \(Q_1,Q_2,Q_3\).
	Let \(I(X)\) denote the interior of the region bounded by the cycle \(X\). Note that
	there is some $i \in \{1,2,3\}$ such that $I(Q_j) \subset I(Q_i)$ for $j \in \{1,2,3\} \setminus \{i\}$. 
	Let  \(R_v = I(Q_i)\) and define the \emph{area} \(A(v)\) of \(R_v\) as the number of terminals contained in \(R_v\), and 
	let \(i(v)\) be the (only) neighbor of \(v\) contained in \(R_v\).

	Let \(v\) be a terminal of \(V(G) \setminus \{u\}\) that minimizes \(A(v)\), 
	and let \(N(v) = \{x,y,z\}\).
	By the minimality of \(A(v)\), we have \(A(v) = 0\), otherwise, there is a terminal \(v'\) such that \(R_{v'}\subseteq R_v\) and \(R_{v'} \neq R_v\),
	and hence $A(v') \leq  A(v)-1$.
	By Subclaim~\ref{claim:degree-1-kite}, \(x,y,z\) have odd degree.
	Suppose, without loss of generality, that \(i(v) = y\).
	Since \(G\) is triangle-free, \(x,z\notin N(y)\), and since \(y\) has odd degree, \(y\) has a neighbor \(y'\) different from \(v\) and \(u'\).
	Now, suppose that there is a path in \(G-u'-y\) joining \(y'\) to \(v\), and let \(P'\) be a shortest such path.
	We may suppose, without loss of generality, that \(P'\) contains \(x\).
	By the minimality of \(P'\), it does not contain \(z\).
	Hence \((yy' + P' + vzu'u, \{yv\},\emptyset)\)  is an FRS for \(G\), a contradiction to Claim~\ref{claim:no-reducing}.
	Thus, we may assume that every path in \(G\) joining \(y'\) to \(v\) contains~\(u'\) or~\(y\).
	Let \(C\) be the component of \(G-u'-y\) containing \(y'\), and let \(G' = G[V(C) \cup \{u', y\}]\).
	Note that \(C\) is contained~\(R_v\).
	We claim that \(G'\) contains a terminal of \(G\).
	Let~\(n'\) be the number of vertices of \(G'\), and let \(t'\) be the number of vertices of \(G'\) with degree at most~\(3\) in \(G'\).
	By Euler's formula, we have \(4n'-8 \geq t' + 4(n'-t')\), hence \(3t' \geq 8\) and \(t'\geq 3\).
	Thus, there is at least one vertex, say \(v'\), in \(G'\) different from \(u'\) and \(y\), with degree at most~\(3\).
	Note that \(v'\in V(C)\subseteq R_v\).
	However, \(d_G(v')=d_{G'}(v') \leq 3\), and hence \(v'\) is a terminal of \(G\), a contradiction to the minimality of \(A(v)\).
\end{proof}

Note that since \(G\) has no terminal with degree \(1\), Claim~\ref{claim:no-useful-cut-vertex} guarantees that \(G\) has no cut-vertex, i.e.,
\(G\) is \(2\)-connected.

\begin{claim}\label{claim:no-terminal-degree-2}
	No terminal has degree \(2\).
\end{claim}

\begin{proof}
	Suppose, for a contradiction, that \(G\) contains a terminal \(u\) with degree \(2\), 
	and
	let~\(v\) be another terminal of \(G\).
	By Claim~\ref{claim:few-small-vertices}, \(d(v)=3\), and hence let \(N(v) = \{x,y,z\}\).
	First, suppose that every neighbor of \(v\) has odd degree.
	Since \(G\) is \(2\)-connected, there is a cycle in \(G\) containing \(u\) and \(v\).
	Let \(C\) be such a cycle that minimizes \(|E(C)|\).
	Suppose, without loss of generality, that \(C\) contains~\(x\) and \(y\).
	By the minimality of \(|E(C)|\), we have \(z\notin V(C)\).
	Thus, \((C - yv+vz, \{yv\},\emptyset)\) is an FRS for \(G\) (see Figure~\ref{fig:m13}), 
	a contradiction to Claim~\ref{claim:no-reducing}.

	Thus, we may suppose that \(v\) has a neighbor with even degree.
	By Claim~\ref{claim:even-degree-neighbors}, there exist at least two vertices in \(N(v)\) with even degree.
	Let \(N(u) = \{a,b\}\).
	Suppose that~\(u\) and~\(v\) have at most one common neighbor, and let \(P\) be a shortest path joining~\(u\) and~\(v\).
	Suppose, without loss of generality, that  \(a,x \in V(P)\) (possibly \(a=x\)) and that~\(z\) has even degree.
	Note that by the minimality of \(P\) we have \(b,y \notin V(P)\) and, by the assumption that~\(u\) and~\(v\) 
	have at most one common neighbor, we have \(b \neq y\).
	Hence, \((bu + P + vy, \{zv\},\emptyset)\) is an FRS for \(G\) (see Figure~\ref{fig:m14}), 
	a contradiction to Claim~\ref{claim:no-reducing}.
 	\begin{figure}[h]
 		\centering
 		\begin{subfigure}[b]{.33\linewidth}
 			\centering\scalebox{.7}{\begin{tikzpicture}[scale = 0.6]

	\node (u) [squared black vertex] at ($(0:0)$) {};
	\node (a) [black vertex] at ($(u)+(0:2.17)+(90:2)$) {};
	\node (b) [black vertex] at ($(u)+(0:2.17)+(-90:2)$) {};

	\node (label_u) [] at ($(u)+(180:.6)$) {$u$};
	\node (label_a) [] at ($(a)+(90:.6)$) {$a$};
	\node (label_b) [] at ($(b)+(-90:.6)$) {$b$};

	\node (v) [squared black vertex] at ($(u)+(0:8)$) {};
	\node (x) [black vertex] at ($(v)+(180:2.17)+(90:2)$) {};
	\node (y) [black vertex] at ($(v)+(-180:2.17)$) {};
	\node (z) [black vertex] at ($(v)+(180:2.17)+(-90:2)$) {};

	\node (label_v) [] at ($(v)+(90:.6)$) {$v$};
	\node (label_x) [] at ($(x)+(0:.6)$) {$x$};
	\node (label_y) [] at ($(y)+(90:.6)$) {$y$};
	\node (label_z) [] at ($(z)+(-90:.6)$) {$z$};

	\draw[edge, color1] (b) -- (u) -- (a) (x) -- (v) (v) -- (z);
	\draw[edge, short snake, color1] (a) to  (x);
	\draw[edge, short snake, color1] (b) to  (y);
	\draw[edge, color2]   (v) -- (y);
\end{tikzpicture}}
 			\caption{}\label{fig:m13}
 		\end{subfigure}%
 		\hfill
 		\begin{subfigure}[b]{.33\linewidth}
 			\centering\scalebox{.7}{\begin{tikzpicture}[scale = 0.6]

	\node (u) [squared black vertex] at ($(0:0)$) {};
	\node (a) [black vertex] at ($(u)+(0:2.17)+(90:2)$) {};
	\node (b) [black vertex] at ($(u)+(0:2.17)+(-90:2)$) {};

	\node (label_u) [] at ($(u)+(180:.6)$) {$u$};
	\node (label_a) [] at ($(a)+(90:.6)$) {$a$};
	\node (label_b) [] at ($(b)+(-90:.6)$) {$b$};

	\node (v) [squared black vertex] at ($(u)+(0:8)$) {};
	\node (x) [black vertex] at ($(v)+(180:2.17)+(90:2)$) {};
	\node (y) [black vertex] at ($(v)+(-180:2.17)$) {};
	\node (z) [black vertex] at ($(v)+(180:2.17)+(-90:2)$) {};

	\node (label_v) [] at ($(v)+(90:.6)$) {$v$};
	\node (label_x) [] at ($(x)+(0:.6)$) {$x$};
	\node (label_y) [] at ($(y)+(90:.6)$) {$y$};
	\node (label_z) [] at ($(z)+(-90:.6)$) {$z$};

        \node (label_path1) [] at ($(u)+(0:4)+(0,2.5)$) {$P$};

	\draw[edge, color1] (b) -- (u) -- (a) (x) -- (v) -- (y);
	\draw[edge, short snake, color1] (a) to  (x);
	\draw[edge, color2] (v) -- (z);
\end{tikzpicture}}
 			\caption{}\label{fig:m14}
 		\end{subfigure}%
 		\hfill
 		\begin{subfigure}[b]{.33\linewidth}
 			\centering\scalebox{.7}{\begin{tikzpicture}[scale = 0.6]

	\node (u) [squared black vertex] at ($(0:0)$) {};

	\node (label_u) [] at ($(u)+(180:.6)$) {$u$};

	\node (v) [squared black vertex] at ($(u)+(0:6)$) {};
	\node (y) [black vertex] at ($(u)+(0:3)$) {};
	\node (x) [black vertex] at ($(y)+(90:2)$) {};
	\node (z) [black vertex] at ($(y)+(-90:2)$) {};

	\node (label_v) [] at ($(v)+(90:.6)$) {$v$};
	\node (label_x) [] at ($(x)+(90:.6)$) {$a = x$};
	\node (label_y) [] at ($(y)+(90:.6)$) {$b = y$};
	\node (label_z) [] at ($(z)+(-90:.6)$) {$z$};

        \node (label_path2) [] at ($(x)+(0:1.7)+(0, -0.5)$) {$P$};

	\draw[edge, color1] (u) -- (x) -- (v) -- (y);
	\draw[edge, short snake, color1] (y) to  (z);
	\draw[edge, color2] (u) -- (y);
	\draw[edge, color2] (z) -- (v);
\end{tikzpicture}}
 			\caption{}\label{fig:m15}
 		\end{subfigure}%
 		\caption{Each figure illustrates a reducing scheme $(H,\Acal,\Lcal)$ where $H$ is illustrated in blue, 
				and the paths in $\mathcal{A}$ are illustrated in red.}
 		\label{fig:case}
 	\end{figure}
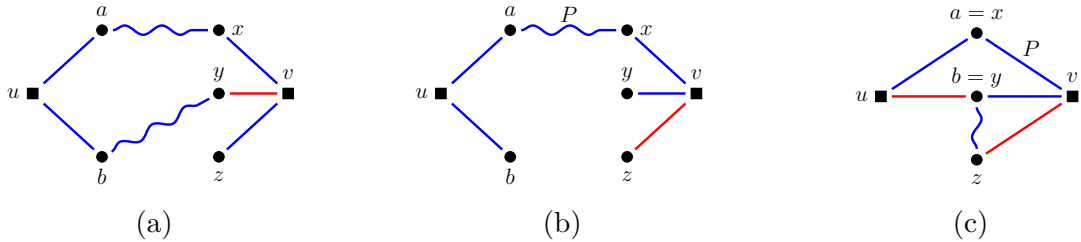

	Therefore, we may assume that \(u\) and \(v\) have precisely two common neighbors, say \(a=x\) and \(b=y\).
	First, suppose that \(z\) has even degree.
	Let \(G' = G -u -v\), and note that, by the minimality of \(G\), \(G'\) is a Gallai graph.
	Let \(\D'\) be a minimum path decomposition of~\(G'\).
	Note that \(z\) has odd degree in \(G'\), and let \(P'_z\) be the path in \(\D'\) having \(z\) as an end vertex.
	Let \(Q = uxvyu\) and \(P_z = P'_z + zv\).
	By Lemma~\ref{lemma:cycle+path}, \(P_z + Q\) can be decomposed into two paths.
	Therefore, there is a path decomposition of $G$ of cardinality $|\D'| + 1 \leq \lfloor n/2 \rfloor$, and hence $G$ is a
	Gallai graph, a contradiction.
	Now, suppose that \(z\) has odd degree, and hence \(x\) and \(y\) have even degree.
	Since \(G\) is 2-connected, there is a path in \(G-v\) joining \(z\) to a vertex in \(\{x,y\}\).
	Let \(P\) be a shortest such path.
	We may assume, without loss of generality, that \(P\) contains \(y\) but does not contain \(x\), also note that \(u \notin V(P)\).
	Hence, \((P + yvxu, \{yu, zv\},\emptyset)\) is an FRS for~\(G\) (see Figure~\ref{fig:m15}), 
	a contradiction to Claim~\ref{claim:no-reducing}.
\end{proof}

By Claim~\ref{claim:no-terminal-degree-1} and~\ref{claim:no-terminal-degree-2}, every terminal of \(G\) has degree \(3\).

\begin{claim}\label{property:three-common-neighbors}
	No two terminals have all three neighbors in common.
\end{claim}

\begin{proof}
	Suppose, for a contradiction, that \(u\) and \(v\) are terminals with \(N(u) = N(v)=\{x,y,z\}\).
	Let \(G' = G-u-v\). By the minimality of \(G\), the graph \(G'\) is a Gallai graph.
	By Claim~\ref{claim:no-terminal-degree-2}, the vertices in \(\{x,y,z\}\) are not isolated in \(G'\).
	Let \(\D'\) be a minimum path decomposition of \(G'\),
	and let \(P'\) be a path in \(\D'\) containing a vertex in \(\{x,y,z\}\).
	We claim that \(P'\) can be splitted into two (possible empty) paths \(Y'\) and \(Z'\) such that \(V(Y') \cap V(Z') = \{x'\}\), \(y' \notin V(Y')\), \(z' \notin V(Z')\), and \(\{x', y', z'\} = \{x, y, z\}\).
	If \(P'\) contains exactly one vertex in $\{x,y,z\}$, say \(x\), then let \(P' = P_1 x P_2\), 
	and note that \(Y' = P_1x\), \(Z' = xP_2\) satisfy the claim with \(x' = x\), \(y' = y\) and \(z'=z\).
	If \(P'\) contains exactly two vertices in $\{x,y,z\}$, say \(x,y\), then let \(P' = P_1 x P_2 y P_3\),  
	and note that \(Y' = P_1x\), \(Z' = xP_2yP_3\) satisfy the claim with \(x'=x\), \(y' = y\) and \(z'=z\).
	Finally, if \(P'\) contains the three vertices in $\{x,y,z\}$, then let \(P' = P_1 x P_2 y P_3 z P_4\), 
	and note that  \(Y' = P_1 x P_2 y\), \(Z' = yP_3zP_4\) satisfy the claim with \(x'=y\), \(y' = z\) and \(z'=x\).
    Therefore, the claim holds.
	Now, let \(Y = Y' + x'uy'v\) and \(Z = Z' + x'vz'u\), 
	and note that \(\big(\D'\cup\{Y,Z\}\big)\setminus\{P'\}\) 
	is a path decomposition of \(G\) (see Figure~\ref{fig:m16}) 
	of cardinality \(|\D'| + 1 \leq\lfloor n/2\rfloor\),
	and hence $G$ is a Gallai graph, a contradiction.
\end{proof}
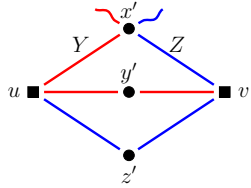
\begin{figure}[h]
	\centering\scalebox{.7}{\begin{tikzpicture}[scale = 0.6]

	\node (u) [squared black vertex] at ($(0:0)$) {};

	\node (label_u) [] at ($(u)+(180:.6)$) {$u$};

	\node (v) [squared black vertex] at ($(u)+(0:6)$) {};
	\node (y) [black vertex] at ($(u)+(0:3)$) {};
	\node (x) [black vertex] at ($(y)+(90:2)$) {};
	\node (z) [black vertex] at ($(y)+(-90:2)$) {};

	\node (anchor_xa) [] at ($(x)+(30:1.5)$) {};
	\node (anchor_xb) [] at ($(x)+(150:1.5)$) {};

	\node (label_v) [] at ($(v)+(0:.6)$) {$v$};
	\node (label_x) [] at ($(x)+(90:.6)$) {$x'$};
	\node (label_y) [] at ($(y)+(90:.6)$) {$y'$};
	\node (label_z) [] at ($(z)+(-90:.6)$) {$z'$};

        \node (label_path1) [] at ($(u)+(1.5,1.5)$) {$Y$};
        \node (label_path1) [] at ($(v)+(-1.5,1.5)$) {$Z$};

	\draw[edge, color1] (x) -- (v) -- (z) -- (u);
	\draw[edge, color2] (x) -- (u) -- (y) -- (v);
	\draw[edge, short snake, color1] (anchor_xa) -- (x);
	\draw[edge, short snake, color2] (anchor_xb) -- (x);
\end{tikzpicture}}
	\caption{Illustration of the proof of Claim~\ref{property:three-common-neighbors}.
	        The paths $Y$ and $Z$ are illustrated, respectively, in red and blue.}\label{fig:m16}
\end{figure}

We say that a separator \(S\) in \(G\) is a \emph{terminal separator} if \(|S|=2\) and $S$ separates two terminals in \(G\).
If \(S\) is a terminal separator and \(C\) is a component of \(G-S\) that contains a terminal of \(G\), then we say that \(C\) is a \emph{terminal component} of \(G-S\).

\begin{claim}\label{claim:terminal-separator}
	Every pair of terminals is separated by a terminal separator.
\end{claim}

\begin{proof}
    Suppose, for a contradiction, that there is no terminal separator that separates the terminals~\(u\) and~\(v\).
    By Claim~\ref{property:three-common-neighbors}, \(u\) and \(v\) have at most two common neighbors.
    In what follows, the proof is divided into two cases, depending on whether \(u\) and \(v\) have at most
    one common neighbor or precisely two common neighbors.
    
    First, suppose that \(u\) and \(v\) have at most one common neighbor.
	Since \(G\) is \(2\)-connected, every separator that separates \(u\) and \(v\) must have size at least~\(3\).
	Thus, there are three internally vertex-disjoint paths, say \(P\), \(Q\), \(R\), joining \(u\) and \(v\).
	Since \(u\) and \(v\) have at most one common neighbor, 
	at most one of the paths \(P\), \(Q\) and \(R\) has length~\(2\).
	Let \(N(u) = \{a,b,c\}\) and \(N(v) = \{x,y,z\}\), where, possibly, \(a=x\).
	Suppose, without loss of generality, that \(a,x\in V(P)\), \(b,y\in V(Q)\), and \(c,z\in V(R)\).
	Suppose that both \(u\) and \(v\) have neighbors with even degree.
	By Claim~\ref{claim:even-degree-neighbors}, we may assume, without loss of generality, that \(c\) and \(z\) have even degree.
	Thus, \((bu + P + vy, \{cu, zv\},\emptyset)\) is an FRS for \(G\) (see Figure~\ref{fig:m25}), 
	a contradiction to Claim~\ref{claim:no-reducing}.
    Thus, we may assume, without loss of generality, that all the neighbors of \(v\) have odd degree.
	Suppose that at least one, say \(c\), in \(\{a,b,c\}\) has even degree.
	Thus, \((P + Q - xv+vz, \{xv, cu\},\emptyset)\) is an FRS for \(G\) (see Figure~\ref{fig:m26}), 
	a contradiction to Claim~\ref{claim:no-reducing}.
	Hence, we may assume that \(a\), \(b\) and \(c\) have odd degree.
	Therefore, \((P + Q + R -au - zv, \{au, zv\},\emptyset)\) is an FRS for \(G\) (see Figure~\ref{fig:m27}), 
	a contradiction to Claim~\ref{claim:no-reducing}.  
    
    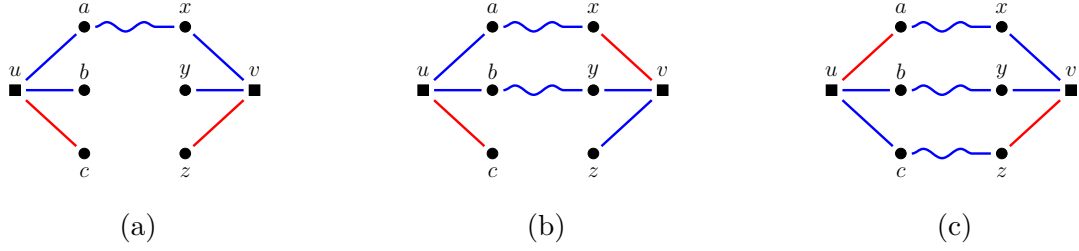
\begin{figure}[h]
    	\centering
    	\begin{subfigure}[b]{.33\linewidth}
    		\centering\scalebox{.7}{\begin{tikzpicture}[scale = 0.6]

	\node (u) [squared black vertex] at ($(0:0)$) {};
	\node (a) [black vertex] at ($(u)+(0:2.17)+(90:2)$) {};
	\node (b) [black vertex] at ($(u)+(0:2.17)$) {};
	\node (c) [black vertex] at ($(u)+(0:2.17)+(-90:2)$) {};

	\node (label_u) [] at ($(u)+(90:.6)$) {$u$};
	\node (label_a) [] at ($(a)+(90:.6)$) {$a$};
	\node (label_b) [] at ($(b)+(90:.6)$) {$b$};
	\node (label_c) [] at ($(c)+(-90:.6)$) {$c$};

	\node (v) [squared black vertex] at ($(u)+(0:7.5)$) {};
	\node (y) [black vertex] at ($(v)+(-180:2.17)$) {};
	\node (x) [black vertex] at ($(y)+(90:2)$) {};
	\node (z) [black vertex] at ($(y)+(-90:2)$) {};

	\node (label_v) [] at ($(v)+(90:.6)$) {$v$};
	\node (label_x) [] at ($(x)+(90:.6)$) {$x$};
	\node (label_y) [] at ($(y)+(90:.6)$) {$y$};
	\node (label_z) [] at ($(z)+(-90:.6)$) {$z$};

	\node (anchor_ya) [] at ($(y)+(-60:1.5)$) {};
	\node (anchor_yb) [] at ($(y)+(-120:1.5)$) {};
	\node (anchor_c) [] at ($(c)+(-45:1.5)$) {};
	\node (anchor_z) [] at ($(z)+(-135:1.5)$) {};
	\node (anchor_x) [] at ($(x)+(0:1.5)$) {};
        

        \draw[edge, short snake, color1]  (a) -- (x);
        \draw[edge, color1]  (u) -- (a)  (v) -- (y) (b) -- (u) (x) -- (v) ;
        \draw[edge, color2] (u) -- (c) ;
        \draw[edge, color2] (v) -- (z) ;

\end{tikzpicture}}
    		\caption{}\label{fig:m25}
    	\end{subfigure}%
    	\hfill
     	\begin{subfigure}[b]{.33\linewidth}
    		\centering\scalebox{.7}{\begin{tikzpicture}[scale = 0.6]

	\node (u) [squared black vertex] at ($(0:0)$) {};
	\node (a) [black vertex] at ($(u)+(0:2.17)+(90:2)$) {};
	\node (b) [black vertex] at ($(u)+(0:2.17)$) {};
	\node (c) [black vertex] at ($(u)+(0:2.17)+(-90:2)$) {};

	\node (label_u) [] at ($(u)+(90:.6)$) {$u$};
	\node (label_a) [] at ($(a)+(90:.6)$) {$a$};
	\node (label_b) [] at ($(b)+(90:.6)$) {$b$};
	\node (label_c) [] at ($(c)+(-90:.6)$) {$c$};

	\node (v) [squared black vertex] at ($(u)+(0:7.5)$) {};
	\node (y) [black vertex] at ($(v)+(-180:2.17)$) {};
	\node (x) [black vertex] at ($(y)+(90:2)$) {};
	\node (z) [black vertex] at ($(y)+(-90:2)$) {};

	\node (label_v) [] at ($(v)+(90:.6)$) {$v$};
	\node (label_x) [] at ($(x)+(90:.6)$) {$x$};
	\node (label_y) [] at ($(y)+(90:.6)$) {$y$};
	\node (label_z) [] at ($(z)+(-90:.6)$) {$z$};

	\node (anchor_ya) [] at ($(y)+(-60:1.5)$) {};
	\node (anchor_yb) [] at ($(y)+(-120:1.5)$) {};
	\node (anchor_c) [] at ($(c)+(-45:1.5)$) {};
	\node (anchor_z) [] at ($(z)+(-135:1.5)$) {};
	\node (anchor_x) [] at ($(x)+(0:1.5)$) {};

        \draw[edge, short snake, color1]  (a) -- (x);
        \draw[edge, short snake, color1]  (b) -- (y);
        \draw[edge, color1]  (u) -- (a)  (v) -- (y) (b) -- (u) (v) -- (z) ;
        \draw[edge, color2] (x) -- (v) ;
        \draw[edge, color2] (u) -- (c) ;


\end{tikzpicture}}
     		\caption{}\label{fig:m26}
    	\end{subfigure}%
    	\hfill
    	\begin{subfigure}[b]{.33\linewidth}
    		\centering\scalebox{.7}{\begin{tikzpicture}[scale = 0.6]

	\node (u) [squared black vertex] at ($(0:0)$) {};
	\node (a) [black vertex] at ($(u)+(0:2.17)+(90:2)$) {};
	\node (b) [black vertex] at ($(u)+(0:2.17)$) {};
	\node (c) [black vertex] at ($(u)+(0:2.17)+(-90:2)$) {};

	\node (label_u) [] at ($(u)+(90:.6)$) {$u$};
	\node (label_a) [] at ($(a)+(90:.6)$) {$a$};
	\node (label_b) [] at ($(b)+(90:.6)$) {$b$};
	\node (label_c) [] at ($(c)+(-90:.6)$) {$c$};

	\node (v) [squared black vertex] at ($(u)+(0:7.5)$) {};
	\node (y) [black vertex] at ($(v)+(-180:2.17)$) {};
	\node (x) [black vertex] at ($(y)+(90:2)$) {};
	\node (z) [black vertex] at ($(y)+(-90:2)$) {};

	\node (label_v) [] at ($(v)+(90:.6)$) {$v$};
	\node (label_x) [] at ($(x)+(90:.6)$) {$x$};
	\node (label_y) [] at ($(y)+(90:.6)$) {$y$};
	\node (label_z) [] at ($(z)+(-90:.6)$) {$z$};

	\node (anchor_ya) [] at ($(y)+(-60:1.5)$) {};
	\node (anchor_yb) [] at ($(y)+(-120:1.5)$) {};
	\node (anchor_c) [] at ($(c)+(-45:1.5)$) {};
	\node (anchor_z) [] at ($(z)+(-135:1.5)$) {};
	\node (anchor_x) [] at ($(x)+(0:1.5)$) {};

        \draw[edge, short snake, color1]  (a) -- (x);
        \draw[edge, short snake, color1]  (b) -- (y);
        \draw[edge, short snake, color1]  (c) -- (z);
        \draw[edge, color1]  (x) -- (v) -- (y) (b) -- (u) -- (c);

        \draw[edge, color2]  (u) -- (a);
        \draw[edge, color2]  (v) -- (z);

\end{tikzpicture}}
    		\caption{}\label{fig:m27}
    	\end{subfigure}%
    	\caption{Each figure illustrates a reducing scheme $(H,\Acal,\Lcal)$ where $H$ is illustrated in blue, 
                 and the paths in $\mathcal{A}$ are illustrated in red.}
    	\label{fig:case4}
    \end{figure}    
    
    Now, suppose that \(u\) and \(v\) have precisely two common neighbors, say~\(x\) and~\(y\), 
    and let \(c\) and~\(z\) be the remaining neighbors of \(u\) and \(v\), respectively.
    Suppose that every vertex in \(N(u)\cup N(v)\) has odd degree. 
    If there is a path \(P\) in \(G - x - y\) joining \(u\) and \(v\), then \((yu + P + vx, \{xu, yv\},\emptyset)\) is an 
    FRS for \(G\) (see Figure~\ref{fig:m21}), a contradiction to Claim~\ref{claim:no-reducing}. 
    Therefore, \(\{x,y\}\) is a terminal separator that separates \(u\) and \(v\), a contradiction.

		Thus, we may assume that \(u\) or \(v\), say \(u\), has a neighbor with even degree.
		By Claim~\ref{claim:even-degree-neighbors}, \(u\) has at least two neighbors with even degree.
		Thus, we may assume that \(x\) has even degree and, again by Claim~\ref{claim:even-degree-neighbors}, 
		the vertex~\(v\) has at least two neighbors with even degree.
		Suppose that \(c\) and  $z$  have even degree.
		Let \(G' = G - u -v\), and hence, by the minimality of \(G\), \(G'\) is a Gallai graph.
		Let \(\D'\) be a minimum path decomposition of \(G'\).
		Note that \(c\) and $z$ have odd degree in \(G'\), and let \(P'_c\) and $P'_z$ be the paths in \(\D'\) 
		having \(c\) and $z$ as end vertices, respectively (possibly $P'_c = P'_z$).
		Put \(Q' = uyvxu\), \(P_c= P'_c + cu\) and \(P_z = P'_z + zv\) (if \(P'_c = P'_z\), then put \(P_c = P_z = P'_c + cu + zv\)).
		By Lemma~\ref{lemma:cycle+path}, \(P_c + Q'\) can be decomposed into two paths, say \(P\) and \(Q\).
		Therefore, \(\D = (\D' \cup \{P, Q, P_z\}) \setminus \{P'_c, P'_z\} \) is a path decomposition of \(G\) 
		\big(if \(P'_c = P'_z\), then put \(\D = (\D'\cup\{P,Q\})\setminus\{P'_c\}\)\big)
		such that \(|\D| = |\D'| + 1 \leq \floor{n/2}\), and hence $G$ is a Gallai graph, a contradiction.

		Thus, we  may assume that at least one between \(c\) and \(z\), say \(c\), has odd degree, and hence, by Claim~\ref{claim:even-degree-neighbors}, \(y\) must have even degree.
		If there is no path of length \(2\) in \(G-u-v\) joining \(c\) to \(y\),
		then \(G' = G - u - v +cy\) has no triangle.
		Note that \(d_{G'}(c) = d_G(c)\) and \(d_{G'}(y) = d_G(y)-1\), hence \(c\) and \(y\) have odd degree in \(G'\).
		Thus, \((uxvz,\{yv\},\{cuy\})\) is an FRS for \(G\) such that \(G_\Hcal = G'\) (see Figure~\ref{fig:m17}),
		a contradiction to Claim~\ref{claim:no-reducing}.
		Now, suppose that there is a path \(cwy\) of length \(2\) in \(G-u-v\) joining \(c\) to \(y\).
		Since \(G\) is triangle-free, \(w \neq x, z\).
		Therefore, \((cwyuxvz, \{cu, yv\},\emptyset)\) is an FRS for \(G\) (see Figure~\ref{fig:m18}), 
		a contradiction to Claim~\ref{claim:no-reducing}.
\end{proof}
	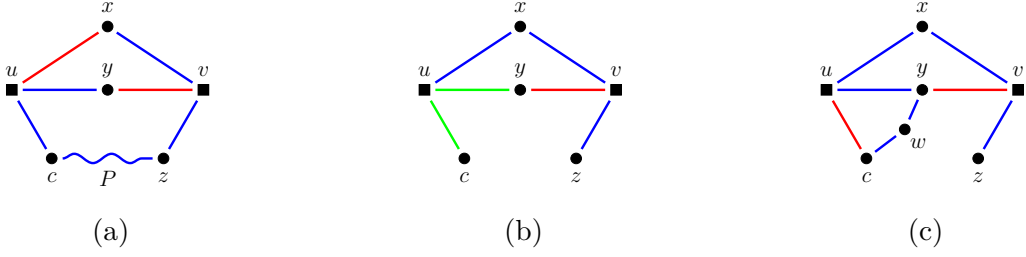
\begin{figure}[h]
		\centering
		\begin{subfigure}[b]{.33\linewidth}
			\centering\scalebox{.7}{\begin{tikzpicture}[scale = 0.6]

	\node (u) [squared black vertex] at ($(0:0)$) {};
	\node (c) [black vertex] at ($(u)+(-60:2.5)$) {};

	\node (label_u) [] at ($(u)+(90:.6)$) {$u$};
	\node (label_c) [] at ($(c)+(-90:.6)$) {$c$};

	\node (v) [squared black vertex] at ($(u)+(0:6)$) {};
	\node (y) [black vertex] at ($(u)+(0:3)$) {};
	\node (x) [black vertex] at ($(y)+(90:2)$) {};
	\node (z) [black vertex] at ($(v)+(240:2.5)$) {};

	\node (label_v) [] at ($(v)+(90:.6)$) {$v$};
	\node (label_x) [] at ($(x)+(90:.6)$) {$x$};
	\node (label_y) [] at ($(y)+(90:.6)$) {$y$};
	\node (label_z) [] at ($(z)+(-90:.6)$) {$z$};

	\node (anchor_ya) [] at ($(y)+(-60:1.5)$) {};
	\node (anchor_yb) [] at ($(y)+(-120:1.5)$) {};
	\node (anchor_c) [] at ($(c)+(-45:1.5)$) {};
	\node (anchor_z) [] at ($(z)+(-135:1.5)$) {};

        \draw[edge, color2] (u) -- (x);
        \draw[edge, color2] (y) -- (v);
        \draw[edge, color1]  (c) -- (u) -- (y) (z) -- (v) -- (x);
        \draw[edge, short snake, color1]  (c) to (z);

     \node (label_Pp) [] at ($(u)+(0:3)+(-90:2.8)$) {$P$};
\end{tikzpicture}}
			\caption{}\label{fig:m21}
		\end{subfigure}		
		\begin{subfigure}[b]{.33\linewidth}
			\centering\scalebox{.7}{\begin{tikzpicture}[scale = 0.6]

	\node (u) [squared black vertex] at ($(0:0)$) {};
	\node (c) [black vertex] at ($(u)+(-60:2.5)$) {};

	\node (label_u) [] at ($(u)+(90:.6)$) {$u$};
	\node (label_c) [] at ($(c)+(-90:.6)$) {$c$};

	\node (v) [squared black vertex] at ($(u)+(0:6)$) {};
	\node (y) [black vertex] at ($(u)+(0:3)$) {};
	\node (x) [black vertex] at ($(y)+(90:2)$) {};
	\node (z) [black vertex] at ($(v)+(240:2.5)$) {};

	\node (label_v) [] at ($(v)+(90:.6)$) {$v$};
	\node (label_x) [] at ($(x)+(90:.6)$) {$x$};
	\node (label_y) [] at ($(y)+(90:.6)$) {$y$};
	\node (label_z) [] at ($(z)+(-90:.6)$) {$z$};

	\node (anchor_ya) [] at ($(y)+(-60:1.5)$) {};
	\node (anchor_yb) [] at ($(y)+(-120:1.5)$) {};
	\node (anchor_c) [] at ($(c)+(-45:1.5)$) {};

        \draw[edge, color1] (u) -- (x) -- (v) -- (z);
        \draw[edge, color3] (y) -- (u) -- (c);
        \draw[edge, color2] (y) -- (v);
\end{tikzpicture}}
			\caption{}\label{fig:m17}
		\end{subfigure}%
		\begin{subfigure}[b]{.33\linewidth}
			\centering\scalebox{.7}{\begin{tikzpicture}[scale = 0.6]

	\node (u) [squared black vertex] at ($(0:0)$) {};
	\node (c) [black vertex] at ($(u)+(-60:2.5)$) {};

	\node (label_u) [] at ($(u)+(90:.6)$) {$u$};
	\node (label_c) [] at ($(c)+(-90:.6)$) {$c$};

	\node (v) [squared black vertex] at ($(u)+(0:6)$) {};
	\node (y) [black vertex] at ($(u)+(0:3)$) {};
	\node (x) [black vertex] at ($(y)+(90:2)$) {};
	\node (w) [black vertex] at ($(u)+(-27:2.75)$) {};
	\node (z) [black vertex] at ($(v)+(240:2.5)$) {};

	\node (label_v) [] at ($(v)+(90:.6)$) {$v$};
	\node (label_x) [] at ($(x)+(90:.6)$) {$x$};
	\node (label_y) [] at ($(y)+(90:.6)$) {$y$};
	\node (label_w) [] at ($(w)+(-45:.6)$) {$w$};
	\node (label_z) [] at ($(z)+(-90:.6)$) {$z$};

	\node (anchor_ya) [] at ($(y)+(-60:1.5)$) {};
	\node (anchor_yb) [] at ($(y)+(-120:1.5)$) {};
	\node (anchor_c) [] at ($(c)+(-45:1.5)$) {};
	\node (anchor_z) [] at ($(z)+(-135:1.5)$) {};

    \draw[edge, color1]  (c) -- (w) -- (y) -- (u) -- (x) -- (v) -- (z);
    \draw[edge, color2] (c) -- (u);
    \draw[edge, color2] (y) -- (v);

\end{tikzpicture}}
			\caption{}\label{fig:m18}
		\end{subfigure}%
		\caption{Each figure illustrates a reducing scheme $(H,\Acal,\Lcal)$ where $H$ is illustrated in blue, 
				the paths in $\mathcal{A}$ are illustrated in red,
				and the paths in $\Lcal$ are illustrated in green.}
		\label{fig:case5}
	\end{figure}

Finally, we have enough structure to conclude the proof of Theorem~\ref{theo:main}. 
By Claim~\ref{claim:six-terminal-vertices}, \(G\) contains at least two terminals and hence, by 
Claim~\ref{claim:terminal-separator}, \(G\) contains a terminal separator.
Given a terminal separator \(S'\), let \(\eta(S')\) be the smallest order (number of vertices) of a terminal component of \(G - S'\).
Let \(S=\{x,y\}\) be a terminal separator of \(G\) that minimizes \(\eta(S)\).
Let \(G'_1,\ldots,G'_k\) be the components of \(G-S\), and let \(G_i = G\big[V(G'_i) \cup S\big]\), for \(i=1,\ldots,k\).
Suppose, without loss of generality, that \(|V(G'_1)| = \eta(S)\) and let \(u\) be a terminal of \(G\) in \(G'_1\), and hence in \(G_1\).
We claim that \(G_1'\) contains another terminal of \(G\).
Indeed, since \(G_1\) is a triangle-free planar graph, we have \(2|E(G_1)| \leq 4|V(G_1)| - 8\).
On the other hand, we have \(2|E(G_1)| = \sum_{v\in V(G_1)} d_{G_1}(v)\).
Let \(R\) be the set of vertices of \(G_1\) different from \(u\) and with degree at most \(3\) in \(G_1\).
We have
\begin{align*}
\sum_{v\in V(G_1)} d_{G_1}(v)
&\quad\geq\quad d_{G_1}(u) + \sum_{v\in R}d_{G_1}(v) + \sum_{v\notin R, v\neq u} d_{G_1}(v) \\
&\quad\geq\quad 3 + |R| + 4(|V(G_1)|-|R|-1) \\
&\quad=\quad 4|V(G_1)| - 3|R| -1
\end{align*}
Thus, we have \(3|R| \geq 7\), and hence \(|R| \geq 3\), which implies that there is a vertex, say \(v\), in~\(R \setminus S\).
Thus, \(d_G(v) = d_{G_1}(v) \leq 3\), and hence \(v\) is a terminal of \(G\).
Since~\(u\) and~\(v\) are two terminals of \(G\), by Claim~\ref{claim:terminal-separator},  there is a terminal separator \(S' = \{x',y'\}\) that separates \(u\) and \(v\).
Since \(u\) and \(v\) are in the same component of \(G-S\), at least one of the vertices of~\(S'\), say~\(x'\), must be in \(G'_1\).
The rest of the proof is divided into two cases, depending on whether \(y'\notin V(G_1')\) or \(y'\in V(G'_1)\).

First, suppose that \(y'\notin V(G_1')\).
Then \(G'_1 - x'\) is disconnected and \(u\) and \(v\) are in different components of \(G'_1 - x'\).
Let \(C_u\) and \(C_v\) be the components of \(G'_1-x'\) containing, respectively, \(u\) and~\(v\).
Note that if \(x\) has neighbors in both \(V(C_u)\) and \(V(C_v)\),
then we can obtain a path \(uP_1xP_2v\), where \(V(P_1)\subseteq V(C_u)\) and \(V(P_2)\subseteq V(C_v)\),
and hence \(y'=x\), otherwise \(S'\) would not separate \(u\) and \(v\).
Analogously, if \(y\) has neighbors in both \(C_u\) and \(C_v\), then \(y'=y\).
Thus, since \(x\neq y\), we may suppose, without loss of generality, that \(y'\neq x\),
and hence \(x\) has no neighbors in both \(C_u\) and \(C_v\).
Suppose that \(x\) has no neighbors in \(C_u\),  and let \(S'' = \{x',y\}\).
Thus, \(C_u\) is a component of \(G-S''\) (see Figure~\ref{fig:final1}).
Since, \(C_u\subseteq G_1'\) and \(v\notin V(C_u)\), we have that \(S''\) is a terminal separator
such that \(\eta(S'') \leq |V(C_u)| < |V(G_1')| = \eta(S)\),
a contradiction to the minimality of \(S\).

\begin{figure}[h]
	\begin{subfigure}[b]{.49\linewidth}
	\centering\scalebox{.7}{\begin{tikzpicture}[scale = 0.7]

    \draw [fill=setfilling, draw=setborder, decoration={random steps,segment length=1cm,amplitude=0cm,
    pre=lineto,pre length=.25cm,post=lineto,post length=.25cm},  decorate, rounded corners=.3cm] (-5,-1) ellipse  ({2cm} and {4.3 cm});   

       \draw [draw=black, fill=yellow, opacity=0.2]
       (-5,-1) -- (-6,0.3) -- (-4.3,.3) -- cycle;

       \draw [draw=black, fill=yellow, opacity=0.2]
       (-5,-1) -- (-6,-2.5) -- (-4.2,-2.5) -- cycle;
       
       \draw [draw=black, fill=yellow, opacity=0.2]
       (0,-3) --  ($(-5,1.5) + (90:1)$) -- ($(-5,1.5)+(-90:1)$) -- cycle;       
       
       \draw [draw=black, fill=yellow, opacity=0.2]
       (0,-3) -- ($(-5,-3.5) + (90:1)$) -- ($(-5,-3.5)+(-90:1)$) -- cycle;   
       
       \draw [draw=black, fill=yellow, opacity=0.2]
       (0,1) -- ($(-5,-3.5) + (135:1)$) -- ($(-5,-3.5)+(-45:1)$) -- cycle;        
    
        \draw [fill=setfilling, draw=setborder, decoration={random steps,segment length=1cm,amplitude=0cm,
        pre=lineto,pre length=.25cm,post=lineto,post length=.25cm},  decorate, rounded corners=.4cm] (-5,1.5) ellipse  ({1cm} and {1cm});  
        
        \draw [fill=setfilling, draw=setborder, decoration={random steps,segment length=1cm,amplitude=0cm,
        pre=lineto,pre length=.25cm,post=lineto,post length=.25cm},  decorate, rounded corners=.4cm] (-5,-3.5) ellipse  ({1cm} and {1cm});		
    
        \draw [draw=black, fill=yellow, opacity=0.2]
       (0,1) -- (5,3.5) -- (5,-5.5) -- cycle;    

        \draw [draw=black, fill=yellow, opacity=0.2]
       (0,-3) -- (5,3.5) -- (5,-5.5) -- cycle;  
    
    \draw [fill=setfilling, draw=setborder, decoration={random steps,segment length=1cm,amplitude=0cm,
    pre=lineto,pre length=.25cm,post=lineto,post length=.25cm},  decorate, rounded corners=.3cm] (5,-1) ellipse ({2cm} and {4.3 cm});

    \node (x) [black vertex] at (0,1) {};
    \node (y) [black vertex] at (0,-3) {};
    \node (x') [black vertex] at (-5,-1) {};
    \node (u) [] at ($(-5,1.3) + (135:1.2)$) {};
    \node (v) [] at ($(-5,-3.3) + (-55:1)$) {};    
  
    \node (label_x) [] at ($(x)+(90:.6)$) {$x$};
    \node (label_y) [] at ($(y)+(90:.6)$) {$y$};
    \node (label_xp) [] at ($(x')+(0:.6)$) {$x'$};
    \node (label_u) [] at ($(-5,1.3)$) {$C_u$};
    \node (label_v) [] at ($(-5,-3.3)$) {$C_v$};
    \node (label_G1) [] at (-3.2,3) {$G_1'$};
    \node (label_G2) [] at (5,-1) {$G-V(G_1)$};    

\end{tikzpicture}}
	\caption{}\label{fig:final1}
	\end{subfigure}
	\begin{subfigure}[b]{.49\linewidth}
	\centering
	\centering\scalebox{.7}{\begin{tikzpicture}[scale = 0.7]

    \draw [fill=setfilling, draw=setborder, decoration={random steps,segment length=1cm,amplitude=0cm,
    pre=lineto,pre length=.25cm,post=lineto,post length=.25cm},  decorate, rounded corners=.3cm] (-5,-1) ellipse  ({2cm} and {4.3 cm});   

       \draw [draw=black, fill=yellow, opacity=0.2]
       (-5.5,-1) -- (-6,0.3) -- (-4.3,.3) -- cycle;

       \draw [draw=black, fill=yellow, opacity=0.2]
       (-5.5,-1)-- (-6,-2.5) -- (-4.2,-2.5) -- cycle;

       \draw [draw=black, fill=yellow, opacity=0.2]
       (-4.5,-1) -- (-6,0.3) -- (-4.3,.3) -- cycle;

       \draw [draw=black, fill=yellow, opacity=0.2]
       (-4.5,-1)-- (-6,-2.5) -- (-4.2,-2.5) -- cycle;
       
       \draw [draw=black, fill=yellow, opacity=0.2]
       (0,1) --  ($(-5,1.5) + (90:1)$) -- ($(-5,1.5)+(-90:1)$) -- cycle;       
       
       \draw [draw=black, fill=yellow, opacity=0.2]
       (0,-3) -- ($(-5,-3.5) + (90:1)$) -- ($(-5,-3.5)+(-90:1)$) -- cycle;      
    
        \draw [fill=setfilling, draw=setborder, decoration={random steps,segment length=1cm,amplitude=0cm,
        pre=lineto,pre length=.25cm,post=lineto,post length=.25cm},  decorate, rounded corners=.4cm] (-5,1.5) ellipse  ({1cm} and {1cm});  
        
        \draw [fill=setfilling, draw=setborder, decoration={random steps,segment length=1cm,amplitude=0cm,
        pre=lineto,pre length=.25cm,post=lineto,post length=.25cm},  decorate, rounded corners=.4cm] (-5,-3.5) ellipse  ({1cm} and {1cm});		
    
        \draw [draw=black, fill=yellow, opacity=0.2]
       (0,1) -- (5,3.5) -- (5,-5.5) -- cycle;    

        \draw [draw=black, fill=yellow, opacity=0.2]
       (0,-3) -- (5,3.5) -- (5,-5.5) -- cycle;  
    
    \draw [fill=setfilling, draw=setborder, decoration={random steps,segment length=1cm,amplitude=0cm,
    pre=lineto,pre length=.25cm,post=lineto,post length=.25cm},  decorate, rounded corners=.3cm] (5,-1) ellipse ({2cm} and {4.3 cm});

    \node (x) [black vertex] at (0,1) {};
    \node (y) [black vertex] at (0,-3) {};
    \node (x') [black vertex] at (-5.5,-1) {};
    \node (y') [black vertex] at (-4.5,-1) {};
    
    \node (u) [] at ($(-5,1.3) + (135:1.2)$) {};
    \node (v) [] at ($(-5,-3.3) + (-55:1)$) {};    
  
    \node (label_x) [] at ($(x)+(90:.6)$) {$x$};
    \node (label_y) [] at ($(y)+(90:.6)$) {$y$};
    \node (label_xp) [] at ($(x')+(180:.6)$) {$x'$};
    \node (label_yp) [] at ($(y')+(0:.6)$) {$y'$};
    \node (label_u) [] at ($(-5,1.3)$) {$C_u$};
    \node (label_v) [] at ($(-5,-3.3)$) {$C_v$};
    \node (label_G1) [] at (-3.2,3) {$G_1'$};
    \node (label_G2) [] at (5.1,-3.9) {$G-V(G_1)$};  
    \node (label_Q) []  at (5.8,-1) {$Q$};
    
    \draw[edge, long snake] (x) to [bend left=15] (5,1.5) -- (5,-3.5)  to [bend left=15] (y);

\end{tikzpicture}}
	\caption{}
	\label{fig:final2}
	\end{subfigure}
	
	\caption{Illustrations of terminal separators.
	        Figure~\ref{fig:final1} illustrate the case where $y'\notin V(G_1')$,
	        while Figure~\ref{fig:final2} illustrate the case where $y'\in V(G_1')$.}
\end{figure}
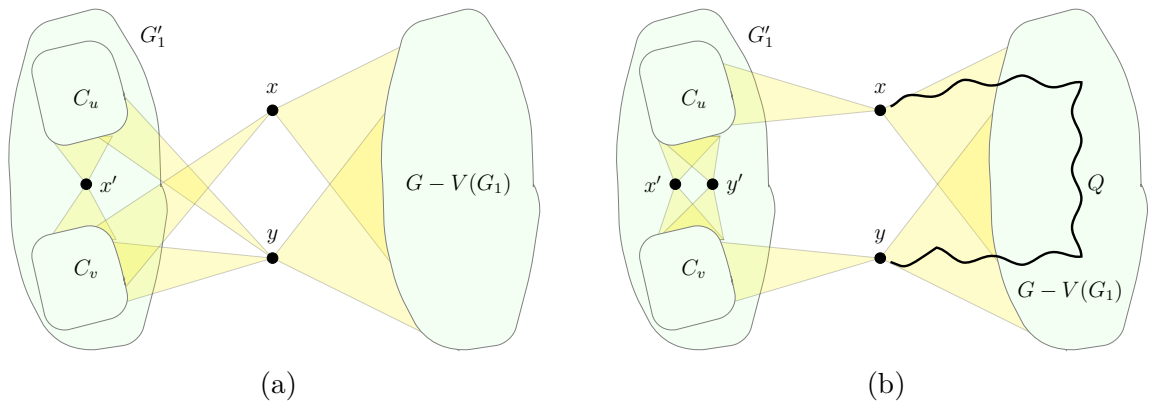

	Finally, suppose that \(y'\in G'_1\).
	Let \(C'_u\) and \(C'_v\) be the components of \(G-S'\) containing, respectively, \(u\) and \(v\).
	Suppose that \(x,y\notin V(C_u')\),
	then \(C_u'\subseteq G - S\), and hence \(C'_u\) is contained in \(G'_1\),
	and since \(v\notin C'_u\), we have \(\eta(S')\leq |V(C'_u)| < |V(G'_1)| = \eta(S)\), a contradiction to the minimality of \(S\).
	Thus, at least one between \(x\) and \(y\) is in \(C'_u\).
	Analogously, at least one between \(x\) and \(y\) is in \(C'_v\).
	Since \(V(C'_u)\cap V(C'_v)=\emptyset\),
	we can suppose that \(x\in V(C'_u)\) and \(y\in V(C'_v)\).
	Since \(G\) is a \(2\)-connected graph, there is a cycle containing \(u\) and any vertex \(z \in V(G'_2)\), and hence there is a path \(Q\) in \(G_2\) joining \(x\) and \(y\) (see Figure~\ref{fig:final2}).
	Since \(S'\subseteq V(G'_1)\), we have \(V(Q) \subseteq G - S'\),
	which implies that \(x\) and \(y\) (and hence \(u\) and \(v\)) 
	are in the same component of \(G-S'\), a contradiction.
	This concludes the proof.
\end{proof}

\section{Concluding remarks}\label{sec:concluding}

The technique used in this paper, as in~\cite{BoSaCoLe}, consists of exploring vertices of degree at most~\(3\) 
in order to reduce a minimal counterexample to a Gallai graph.
Although maximal triangle-free planar graphs have less edges than maximal partial \(3\)-trees,
this paper required a deeper and more powerful technique with new elements,
an extension of the previously introduced reducing subgraphs,
which we call feasible reducing schemes.
We believe that this technique may be extended even further in order to prove Gallai's Conjecture for planar graphs.
For that, one must develop a way of dealing either with triangles, or with vertices with degrees \(4\) and \(5\) 
(recall that every planar graph contains at least four vertices with degree at most \(5\)).

Also, the results presented here and in~\cite{BoSaCoLe} suggest that a stronger statement may hold,
i.e., that a graph \(G\) is either a Gallai graph or \(|E(G)| > \lfloor |V(G)|/2\rfloor (|V(G)|-1)\).
This problem was formalized by Bonamy and Perret~\cite{BonamyPerrett16+}.

\bibliographystyle{amsplain}
\bibliography{bibliografia}

\begin{bibdiv}
\begin{biblist}

\bib{BonamyPerrett16+}{article}{
      author={{Bonamy}, Marthe},
      author={{Perrett}, Thomas~J.},
       title={{Gallai's path decomposition conjecture for graphs of small
  maximum degree}},
        date={2016-09},
     journal={ArXiv e-prints},
      eprint={1609.06257},
}

\bib{Bondy14}{article}{
      author={Bondy, Adrian},
       title={Beautiful conjectures in graph theory},
        date={2014},
        ISSN={0195-6698},
     journal={European J. Combin.},
      volume={37},
       pages={4\ndash 23},
         url={http://dx.doi.org/10.1016/j.ejc.2013.07.006},
      review={\MR{3138588}},
}

\bib{BoSaCoLe}{article}{
      author={Botler, F.},
      author={Sambinelli, M.},
      author={Coelho, R.~S.},
      author={Lee, O.},
       title={On {G}allai's and {H}aj\'os' conjectures for graphs with
  treewidth at most \(3\)},
        date={2017-06},
     journal={ArXiv e-prints},
      eprint={1706.04334},
        note={submitted},
}

\bib{BotlerJimenez2015+}{article}{
      author={Botler, F\'abio},
      author={Jim\'enez, Andrea},
       title={On path decompositions of {$2k$}-regular graphs},
        date={2017},
        ISSN={0012-365X},
     journal={Discrete Math.},
      volume={340},
      number={6},
       pages={1405\ndash 1411},
      review={\MR{3624626}},
}

\bib{Di10}{book}{
      author={Diestel, Reinhard},
       title={Graph theory},
     edition={Fourth},
      series={Graduate Texts in Mathematics},
   publisher={Springer},
     address={Heidelberg},
        date={2010},
      volume={173},
        ISBN={978-3-642-14278-9},
         url={http://dx.doi.org/10.1007/978-3-642-14279-6},
      review={\MR{2744811 (2011m:05002)}},
}

\bib{Fan05}{article}{
      author={Fan, Genghua},
       title={Path decompositions and {G}allai's conjecture},
        date={2005},
        ISSN={0095-8956},
     journal={J. Combin. Theory Ser. B},
      volume={93},
      number={2},
       pages={117\ndash 125},
         url={http://dx.doi.org/10.1016/j.jctb.2004.09.008},
      review={\MR{2117933 (2005m:05125)}},
}

\bib{FavaronKouider88}{article}{
      author={Favaron, Odile},
      author={Kouider, Mekkia},
       title={Path partitions and cycle partitions of {E}ulerian graphs of
  maximum degree {$4$}},
        date={1988},
        ISSN={0081-6906},
     journal={Studia Sci. Math. Hungar.},
      volume={23},
      number={1-2},
       pages={237\ndash 244},
      review={\MR{962453}},
}

\bib{GengFangLi15}{article}{
      author={Geng, Xianya},
      author={Fang, Minglei},
      author={Li, Dequan},
       title={Gallai's conjecture for outerplanar graphs},
        date={2015},
     journal={Journal of Interdisciplinary Mathematics},
      volume={18},
      number={5},
       pages={593\ndash 598},
      eprint={http://dx.doi.org/10.1080/09720502.2014.1001570},
         url={http://dx.doi.org/10.1080/09720502.2014.1001570},
}

\bib{JimenezWakabayashi2017}{article}{
      author={Jim\'enez, Andrea},
      author={Wakabayashi, Yoshiko},
       title={On path-cycle decompositions of triangle-free graphs},
        date={2017},
        ISSN={1365-8050},
     journal={Discrete Math. Theor. Comput. Sci.},
      volume={19},
      number={3},
       pages={Paper No. 7, 21},
      review={\MR{3717034}},
}

\bib{Lovasz68}{incollection}{
      author={Lov{\'a}sz, L.},
       title={On covering of graphs},
        date={1968},
   booktitle={Theory of {G}raphs ({P}roc. {C}olloq., {T}ihany, 1966)},
   publisher={Academic Press, New York},
       pages={231\ndash 236},
      review={\MR{0233723 (38 \#2044)}},
}

\bib{Pyber96}{article}{
      author={Pyber, L.},
       title={Covering the edges of a connected graph by paths},
        date={1996},
        ISSN={0095-8956},
     journal={J. Combin. Theory Ser. B},
      volume={66},
      number={1},
       pages={152\ndash 159},
         url={http://dx.doi.org/10.1006/jctb.1996.0012},
      review={\MR{1368522 (96i:05101)}},
}

\end{biblist}
\end{bibdiv}

\end{document}